\documentclass[11pt, letterpaper]{amsart}
\usepackage[active]{srcltx}
\usepackage{calc,amssymb,amsthm,amsmath, ulem}
\usepackage{alltt}
\RequirePackage[dvipsnames,usenames]{color}

\input{kmacros3.sty}
\input{xy}
\xyoption{all}

\DeclareMathOperator{\Id}{{Id}}
\newcommand{\F}{\mathbb{F}}
\renewcommand{\:}{\colon}
\newcommand{\cf}{{\itshape cf.} }
\newcommand{\eg}{{\itshape e.g.} }
\newcommand{\ie}{{\itshape i.e.} }
\newcommand{\etale}{{\'e}tale }
\renewcommand{\m}{\mathfrak{m}}
\DeclareMathOperator{\Tr}{Tr}
\newcommand{\Tt}{{\mathfrak{T}}}
\newcommand{\RamiT}{\mathcal{R}_{\Tt}}
\DeclareMathOperator{\length}{length}
\DeclareMathOperator{\Fitt}{Fitt}

\newcommand{\Frob}[2]{{#1}^{1/p^{#2}}}
\newcommand{\Frobp}[2]{{(#1)}^{1/p^{#2}}}
\newcommand{\FrobP}[2]{{\left(#1\right)}^{1/p^{#2}}}
\newcommand{\bp}{\mathfrak{p}}
\newcommand{\bq}{\mathfrak{q}}
\newcommand{\Ram}{\mathrm{Ram}}

\usepackage[left=1.15in,top=1.06in,right=1.15in,bottom=1.06in]{geometry}

\usepackage{setspace}
\usepackage{hyperref}

\usepackage{enumerate}

\usepackage{graphicx}

\usepackage[all,cmtip]{xy}
\usepackage{verbatim}

\renewcommand{\O}{\mathcal O}

\begin{document}

\title[Test ideals and finite morphisms]{On the behavior of test ideals under finite morphisms}
\author{Karl Schwede and Kevin Tucker}

\address{Department of Mathematics\\ The Pennsylvania State University\\ University Park, PA, 16802, USA}
\email{schwede@math.psu.edu}
\address{Department of Mathematics\\ University of Utah\\ Salt Lake City, UT, 84112, USA}
\email{kevtuck@umich.edu}

\thanks{The first author was partially supported by a National Science Foundation postdoctoral fellowship, RTG grant number 0502170 and NSF DMS 1064485/0969145}
\thanks{The second author was partially supported by RTG grant number 0502170 and a National Science Foundation postdoctoral fellowship DMS 1004344}

\begin{abstract}
We derive precise transformation rules for test ideals under an arbitrary finite surjective morphism $\pi \: Y \to X$ of normal varieties in prime characteristic $p > 0$.  Specifically, given a \mbox{$\bQ$-divisor} $\Delta_{X}$ on $X$ and any $\O_{X}$-linear map $\Tt \: K(Y) \to K(X)$, we associate a \mbox{$\bQ$-divisor} $\Delta_{Y}$ on $Y$ such that $\Tt ( \pi_{*}\tau(Y;\Delta_{Y})) = \tau(X;\Delta_{X})$.  When $\pi$ is separable and $\Tt = \Tr_{Y/X}$ is the field trace, we have $\Delta_{Y} = \pi^{*} \Delta_{X} - \Ram_{\pi}$ where $\Ram_{\pi}$ is the ramification divisor.  If in addition $\Tr_{Y/X}(\pi_{*}\O_{Y}) = \O_{X}$, we conclude $\pi_{*}\tau(Y;\Delta_{Y}) \cap K(X) = \tau(X;\Delta_{X})$ and thereby recover the analogous transformation rule to multiplier ideals in characteristic zero.  Our main technique is a careful study of when an $\O_{X}$-linear map $F_{*} \O_{X} \to \O_{X}$ lifts to an $\O_{Y}$-linear map
$F_{*} \O_{Y} \to \O_{Y}$, and the results obtained about these liftings are of independent interest
as they relate to the theory of Frobenius splittings.  In particular, again assuming $\Tr_{Y/X}(\pi_{*}\O_{Y}) = \O_{X}$, we obtain transformation results for $F$-pure singularities under finite maps which mirror those for log canonical singularities in characteristic zero.   Finally we explore new conditions on the singularities of the ramification locus, which imply that, for a finite extension of normal domains $R \subseteq S$ in characteristic $p > 0$, the trace map $\Tr : \Frac S \to \Frac R$ sends $S$ onto $R$.

\end{abstract}
\maketitle
\numberwithin{equation}{theorem}

\section{Introduction}
\label{sec:Introduction}

The test ideal $\tau(R)$ of a commutative ring $R$ of prime
characteristic $p>0$ is an important measure of singularities and was first described as the set of test elements in the celebrated theory of tight closure pioneered by Hochster and Huneke \cite{HochsterHunekeTC1, HunekeTightClosureBook}.  Following the extension of test ideals to the study of singularities of pairs
 \cite{HaraYoshidaGeneralizationOfTightClosure,TakagiInterpretationOfMultiplierIdeals}, more recent interest in test ideals largely stems from their connections to analytic constructions in complex algebraic geometry.
The multiplier ideal $\mJ(X; \Delta)$ of a $\Q$-divisor $\Delta$ on a
complex algebraic variety $X$ is a critical tool in the study of higher dimensional birational geometry:
surprisingly, after reduction to characteristic $p \gg 0$, the multiplier ideal and the test ideal coincide \cite{SmithMultiplierTestIdeals,HaraInterpretation,HaraYoshidaGeneralizationOfTightClosure,TakagiInterpretationOfMultiplierIdeals}.

This somewhat mysterious correspondence is a great source of intuition and has inspired (and been used to prove) numerous results about both test ideals and multiplier ideals alike, \eg \cite{TakagiFormulasForMultiplierIdeals,BlickleMustataSmithDiscretenessAndRationalityOfFThresholds}.
Nevertheless, these objects are inherently very different.  For
example, many properties of multiplier ideals which follow immediately
from their description via a log resolution of singularities are often
difficult to prove (and occasionally even false) for test ideals, \eg
\cite{BlickleMustataSmithDiscretenessAndRationalityOfFThresholds,MustataYoshidaTestIdealVsMultiplierIdeals}.
Conversely, deep theorems about multiplier ideals requiring subtle
vanishing theorems are often very easy to show for test ideals, \eg \cite{TakagiFormulasForMultiplierIdeals}.

In this paper, we derive transformation rules for test ideals under
arbitrary finite surjective morphisms of normal varieties.
Previous results \cite{HochsterHunekeFRegularityTestElementsBaseChange, BravoSmithBehaviorOfTestIdealsUnderSmooth,HaraYoshidaGeneralizationOfTightClosure,HaraTakagiOnAGeneralizationOfTestIdeals}
apply only in very limited special cases, generally requiring
morphisms that are \etale in codimension one together with further
restrictions.  As finite integral ring extensions have featured
prominently in the development of tight closure, the transformation
rules obtained herein also fit naturally into this largely algebraic theory -- even though their statements (and motivation) are more geometric.

Before describing our results in full detail, we briefly review the behavior of the multiplier ideal under a finite surjective morphism
$\pi \: Y \to X$ of normal complex varieties.  If $\Delta_{X}$ is a $\Q$-divisor on $X$ such that $K_X + \Delta_{X}$ is $\bQ$-Cartier, one relates the multiplier ideal of $(X, \Delta_X)$ with the multiplier ideal of an appropriately defined pair $(Y, \Delta_Y)$.  Explicitly, if $\Ram_{\pi}$ denotes the ramification divisor of $\pi$, then $\Delta_{Y} := \pi^* \Delta_X - \Ram_{\pi}$ is chosen so as to force the equality $\pi^*(K_X + \Delta_X) = K_Y + \Delta_Y$.  Standard techniques relying on resolution of singularities can then be used to show
\begin{equation}
\label{eqn:multIdealFact}
\pi_* \mJ(Y; \Delta_Y) \cap \C(X) = \mJ(X; \Delta_{X})
\end{equation}
where $\C(X)$ is the function field of $X$.  In other words if $X = \Spec R$, then $\C(X)$ is simply the fraction field of $R$.
Proofs and further discussion can be found in either \cite[Theorem 9.5.42]{LazarsfeldPositivity2} or \cite[Proposition 2.8]{EinMultiplierIdealsVanishing}.  It is worth mentioning that this transformation rule does not hold for the multiplier ideal in characteristic $p > 0$, see Examples \ref{ex:WithoutTraceSurjectiveThingsAreFalse}, \ref{ex:artinsD4examplecover} and the surrounding discussion.

One might expect an analogous relation to Equation (\ref{eqn:multIdealFact}) to hold for test ideals when the map $\pi \: Y \to X$ is separable (\ie  surjective and generically {\'e}tale, so that the induced inclusion of function fields $K(X) \subseteq K(Y)$ is separable).
However, this hope is easily shown to be false;  see Example \ref{ex:WithoutTraceSurjectiveThingsAreFalse}.  Essentially, arithmetic problems arise in the presence
of wild ramification in positive characteristic.
In the result below, we recover the analogue of \eqref{eqn:multIdealFact} under a tameness condition for the ramification appearing in $\pi$.

\vskip 6pt
\hskip -12pt
{\bfseries Corollary  \ref{cor:IfTraceSurjectiveThenIntersection}}{ \itshape
Suppose that $\pi \: Y \to X$ is a finite separable morphism of normal
$F$-finite varieties in characteristic $p>0$ and that $\Delta_{X}$ is
a $\Q$-divisor on $X$.  Assume that the trace map $\Tr_{Y/X} \: K(Y)
\to K(X)$ on the function fields satisfies $\Tr_{Y/X}(\pi_{*}\O_{Y}) = \O_{X}$.  If $\Delta_{Y} = \pi^{*}\Delta_{X} - \Ram_{\pi}$ where $\Ram_{\pi}$ is the ramification divisor, then
\[
   \pi_* \left(\tau(Y; \Delta_{Y})\right) \cap K(X) = \tau(X; \Delta_X).
\]
The natural generalization of this result holds for the test ideals of triples $(X, \Delta, \ba^t)$ as well.}
\vskip 6pt

For a separable morphism of normal varieties $\pi \: Y \to X$, the
(non-zero) fraction field trace $\Tr_{Y/X} \: K(Y) \to K(X)$ always satisfies $\Tr_{Y/X}(\pi_{*}\O_{Y}) \subseteq \O_{X}$.  In other words, one may consider trace as a map of sheaves $\Tr_{Y/X} \: \pi_{*}\O_{Y} \to \O_{X}$, and thus we are apt to refer to the condition $\Tr_{Y/X}(\pi_{*}\O_{Y}) = \O_{X}$ in the above result simply by saying that the trace map is surjective.  In our main result stated below, we show how one may overcome the failure of this tameness condition by incorporating the trace map into the transformation rule itself.  Surprisingly, we are able to show a similar transformation rule for the test ideal under arbitrary (\ie not necessarily separable) surjective finite morphisms.

\begin{mainthm}
Assume that $\pi \: Y \to X$ is a finite surjective morphism of normal $F$-finite varieties in characteristic $p > 0$ and that $\Delta_X$ is a $\bQ$-divisor on $X$.
\begin{description}
\item[(Separable Case, Corollary \ref{cor:TraceTestIdealFormula})]  Suppose that $\pi$ is separable and $\Tr_{Y/X} \: K(Y) \to K(X)$ is the trace map.   If $\Delta_Y = \pi^* \Delta_X - \Ram_{\pi}$ where $\Ram_{\pi}$ is ramification divisor, then
    \[
        \Tr_{Y/X} \left( \pi_* \tau(Y; \Delta_{Y})\right) = \tau(X; \Delta_X).
    \]
\item[(General Case, Theorem \ref{thm:TraceTestIdealFormula})]  For an arbitrary $\pi$, any non-zero $\O_{X}$-linear map $\Tt \: K(Y) \to K(X)$ between the function fields of $Y$ and $X$ naturally corresponds to a (possibly non-effective) divisor $\RamiT$ on $Y$ linearly equivalent to $K_Y - \pi^* K_X$.  If $\Delta_{Y/\Tt} = \pi^* \Delta_X - \RamiT$, then
\[
        \Tt\left( \pi_* \tau(Y; \Delta_{Y/\Tt})\right) = \tau(X; \Delta_X).
    \]
\end{description}
The natural generalizations of these results hold for the test ideals of triples $(X, \Delta, \ba^t)$ as well.
\end{mainthm}

\noindent
See Section \ref{sec:CanModDuality} for the construction of the divisor $\RamiT$ noting that in the language of that section, $\sL = \O_X$.  In fact, the Separable Case is just a
specialization of the General Case where $\Tt := \Tr$ and thus $\RamiT = \Ram_{\pi}$; see Section \ref{sec:traceramifdiv}.

Notice that, in light of the role of the trace map in the Separable Case of the Main Theorem, surjectivity of trace is a natural hypothesis in Corollary \ref{cor:IfTraceSurjectiveThenIntersection}.
Conversely, in
Section~\ref{sec:SurjectivityOfTrace} we will explore some conditions which imply that the trace is surjective.  For example, consider the following result:
\vskip 6pt
\hskip -12pt
{\bf Corollary \ref{cor:EasyTraceSurjectivityViaFPurity}. }  {\it   Suppose that $\pi \: Y \to X$ is a finite separable map of $F$-finite normal integral schemes which is tamely ramified in codimension one.  Further suppose that $X$ is regular and that the reduced branch locus of $\pi$ is a divisor with $F$-pure singularities (for example, if the branch locus is a simple normal crossings divisor).  Then $\Tr_{Y/X}(\pi_* \O_Y) = \O_X$.}
\vskip 6pt
\noindent
Substantially more general variants of this result appear in Theorem \ref{thm:surjectivityOfTraceViaTestIdeals}.

Another important result appearing herein analyzes the behavior under finite morphisms of sharply $F$-pure singularities, which are a positive characteristic analog of log canonical singularities \cite{HaraWatanabeFRegFPure}.
Once again, the surjectivity of the trace map plays an essential role.  Briefly, in addition to the set up of Corollary \ref{cor:IfTraceSurjectiveThenIntersection},
assume further that $K_X + \Delta_X$ is $\bQ$-Cartier with index not divisible by $p$.  Then we show that $(X, \Delta_X)$ is sharply $F$-pure if and only if $(Y, \Delta_Y)$ is sharply $F$-pure; see Theorem \ref{thm:TraceSurjectiveImpliesFPureBehaves}.  This result should be compared with transformation rules for log canonical singularities under finite morphisms in characteristic zero, \cf \cite[Proposition 5.20]{KollarMori}, \cite[Proposition 20.2]{KollarFlipsAndAbundance}, and \cite[Proposition 3.16]{KollarSingularitiesOfPairs}.

The proofs of the results described above follow from the study of two fundamental and purely algebraic questions of independent interest.
Suppose that $R \subseteq S$ is a module-finite inclusion of normal $F$-finite domains with induced map of schemes $\pi \: Y = \Spec S \to X = \Spec R$, and let $K := \Frac(R)$ and $L := \Frac(S)$ denote the corresponding fraction fields.
\begin{enumerate}
\item  Does an $R$-linear map $\phi \: R^{1/p^e} \to R$ extend to an $S$-linear map $\bar\phi \: S^{1/p^e} \to S$, where $R^{1/p^{e}}$ (respectively $S^{1/p^e}$) is the ring of $p^{e}$-th roots of $R$ (respectively $S$) inside an algebraic closure of its fraction field?
\item  Given a non-zero $K$-linear map $\Tt : L \to K$, for which $R$-linear maps $\phi \: R^{1/p^e} \to R$ does there exist an $S$-linear map $\phi_{\Tt} \: S^{1/p^e} \to S$ such that $\Tt \circ \phi_{\Tt} = \phi \circ \Tt^{1/p^e}$?
\end{enumerate}
The maps $\phi \in \Hom_R(R^{1/p^e}, R)$ (somewhat abusively referred
to as \mbox{$p^{-e}$-linear}) in the above questions appear frequently in commutative algebra in the study of singularities via Frobenius techniques.  In particular, these maps are used to define
strong $F$-regularity, sharp $F$-purity, and the test ideal.  Important examples of such maps include Frobenius splittings (usually referred to as \emph{$F$-splittings}), which have historically been very useful in applying positive characteristic methods in representation theory; see \cite{BrionKumarFrobeniusSplitting}.

 In Section \ref{sec:LiftingCriterionTheorem}, we will give a precise answer to question (1) when $\pi$ is separable by associating certain geometric data. An oft-used duality argument assigns to each $R$-linear map $\phi \: R^{1/p^e} \to R$ a $\bQ$-divisor $\Delta_{\phi}$ such that $K_X + \Delta_{\phi} \sim_{\bQ} 0$ (compare with \cite{MehtaRamanathanFrobeniusSplittingAndCohomologyVanishing}, \cite{HaraWatanabeFRegFPure}, \cite{SchwedeFAdjunction}, and Section \ref{sec:CanModDuality} of this paper). 
The main technical tool used to prove Corollary~\ref{cor:IfTraceSurjectiveThenIntersection} and Theorem~\ref{thm:TraceSurjectiveImpliesFPureBehaves} is the following extension criterion for $\phi \: R^{1/p^{e}} \to R$ to $\bar{\phi} \: S^{1/p^{e}} \to S$ in terms of the associated divisor $\Delta_{\phi}$.
\vskip 6pt
\hskip -12pt
{\bf Theorem \ref{thm:mainliftingcriterion}. } {\it Using the above notation and assuming that $\pi$ is separable,  an $R$-linear map $\phi \: R^{1/p^e} \to R$ extends to an $S$-linear map $\bar{\phi} \:  S^{1/p^e} \to S$ if and only if $\pi^* \Delta_{\phi} \geq \Ram_{\pi}$.  Furthermore, in this case, one has $\Delta_{\bar{\phi}} = \pi^{*}\Delta_{\phi} - \Ram_{\pi}$ (\ie $K_{Y} + \Delta_{\bar{\phi}} = \pi^{*}(K_{X} + \Delta_{\phi})$).
}
\vskip 6pt
Critical to the proof of this result are two important properties of the trace map.
The first is, as stated above, that the trace map itself corresponds to the ramification divisor under the duality arguments used throughout; see Section \ref{sec:traceramifdiv}.  The second is a certain kind of functoriality property satisfied by the trace map.  Specifically, if $\phi$ extends to $\bar{\phi}$ as in question~(1), then $\Tr_{Y/X} \, \circ \,  \bar{\phi} = \phi \, \circ \, \Tr_{Y/X}^{1/p^e}$.  In private communications after the completion of this work, we learned that this commutativity has also been recently and independently observed by David Speyer in the study of a somewhat different (albeit related) question.

In fact, we are surprisingly able to show that this commutativity property largely characterizes the trace map; see Corollary \ref{cor:trcommutes} and Proposition \ref{prop.TraceIsOnlyMap}. In particular, questions (1) and (2) are intimately linked when $\Tt = \Tr_{Y/X}$.  The answer to question (2) for general $\Tt$ incorporates similar themes, and is the central ingredient in the proof of the General Case of the Main Theorem.

\begin{remark}
An earlier draft of this paper, which was posted on the arXiv, only
dealt with separable extensions.  The methods for the \emph{general case} are
similar and have therefore been included herein.  This
version of the paper also removes two appendices that appeared in previous versions.  It removes an appendix dealing with the behavior of
multiplier ideals in positive characteristic.  This material has been incorporated into
a separate paper \cite{BlickleSchwedeTuckerTestAlterations}, joint with Manuel Blickle.  The other appendix, on explicitly lifting finite maps under tame ramification, will become a separate work \cite{SchwedeTuckerExplicitLiftingOfFSplittings}.  Both appendixes can be viewed in older versions of the arXiv version of this paper.
\end{remark}

\vskip 12pt
\hskip -12pt{\it Acknowledgements:}  The authors would like to thank
Luchezar Avramov, Matthew Baker, Brian Conrad, and Joseph Lipman for
useful discussions on the relation between the trace map and the
ramification divisor,  Alexander Schmidt for pointing out his recent
work on various notions of tameness, as well as  Manuel Blickle, Dale Cutkosky,
Lawrence Ein, Steven Kleiman, J\'anos Koll\'ar, Mircea Musta{\c{t}}{\u{a}}, Karen
Smith and David Speyer for stimulating discussions on other aspects of the paper.  The
authors are also grateful to Wenliang Zhang and the referee for a careful reading of a
previous draft and for pointing out several typos.  Substantial
progress on the questions of this paper was made while attending the Pan-American Advanced Study Institute conference honoring Wolmer Vasconcelos in Olinda Brazil, August 2009.

\section{Notation, divisors, and $p^{-e}$-linear maps}
\label{sec:MapsDivisorsNotation}

\subsection{Basic conventions and notation}
\label{sec:conventions}

Throughout our investigations, all rings are assumed commutative with identity, Noetherian and excellent, and all schemes Noetherian and separated.  If $U = \Spec A$ is an affine scheme, we will repeatedly conflate algebraic and geometric notation by making no distinction between an $A$-module $M$ and the associated  quasi-coherent sheaf on $\Spec A$.  For instance, if $\rho \: U = \Spec B \to V = \Spec A$ is a morphism of affine schemes, we shall use $\rho^{*}M$ to denote $B \tensor_{A}M$; similarly, restricting the scalars on a $B$-module $N$ yields the $A$-module $\rho_{*}N$.

If $R$ is a ring with characteristic $p > 0$, the Frobenius or $p$-th power map can be thought of in several equivalent ways.  First and foremost, it is the ring homomorphism $F \: R \to R$ given by $r \mapsto r^{p}$.  As such, it induces an endomorphism $F \: X = \Spec R \to X$.  By the previous discussion applied to an $R$-module $M$, the $R$-module $F_{*}M$ has the same underlying Abelian group as $M$ but with scalar multiplication twisted by the Frobenius map.  In particular, we may consider the Frobenius map as a $R$-module homomorphism $F \: R \to F_{*}R$.  Similar discussions apply to the $e$-th iterate $F^{e}$ of Frobenius.  Furthermore, when $R$ is a domain we can form the ring of $p^{e}$-th roots $R^{1/p^{e}} = \{ r^{1/p^{e}} : r \in R \}$.  The inclusion $R \subseteq R^{1/p^{e}}$ is naturally identified with the $e$-th iterate of the Frobenius map after identifying $R^{1/p^{e}}$ with $F^{e}_{*}R$.  We remark that the strength and weakness of this latter notation lies in its ability to distinguish between $r \in R$ and $r^{1/p^{e}} \in  R^{1/p^{e}} \cong F^{e}_{*}R$.

We also recall the following definition.

\begin{definition}
\label{DefnFFinite} A scheme $X$ of prime characteristic $p>0$ is \emph{$F$-finite} if
the Frobenius endomorphism $F \: X \to X$ is finite.  In other words, $F_*\O_X$ is a finitely generated or coherent sheaf of $\O_{X}$-modules.  A ring $R$ is called $F$-finite if $X = \Spec R$ is $F$-finite.
\end{definition}

\noindent
  \begin{tabular}[c]{ll}
    {\scshape Convention:} & \parbox[t]{4.5in}{Throughout this paper,
      all characteristic $p>0$ schemes are
  assumed to be $F$-finite.}\\
  \end{tabular}

\vskip 6pt
\smallskip

Most rings and schemes arising geometrically (our primary interest) automatically satisfy the above assumption, as a scheme essentially of finite type over a perfect field is always $F$-finite.
Note that an $F$-finite scheme is locally excellent
\cite{KunzOnNoetherianRingsOfCharP}, and an $F$-finite affine scheme always has a dualizing complex \cite[Remark 13.6]{Gabber.tStruc}.

\subsection{Divisors}
\label{sec:divisors}

Fix a normal  irreducible scheme $X$, with function field $K.$  A \emph{$\bQ$-Divisor} $D$ is an element of  Div($X$) $\tensor_{\bZ} \bQ$.
Recall that $D$ is
\emph{$\bQ$-Cartier} in case there exists an integer $m > 0$ such that $mD$ is a Cartier or locally principal divisor (in particular, $mD$ is necessarily an integral divisor).  For any $\bQ$-divisor $D = \sum a_i D_i$ (where $D_i$ are distinct prime divisors), we use the notation $\lceil D \rceil$ to denote $\sum \lceil a_i \rceil D_i$.  Similarly, set $\lfloor D \rfloor = \sum \lfloor a_{i} \rfloor D_{i}$.  The $\bQ$-divisor  $D$ is \emph{effective} if all of its coefficients $a_{i}$ are non-negative, and we  say that $D' \leq D$ if $D - D'$ is effective.

Let us now describe how to pull-back divisors under a finite map, \cf \cite[Proof of Proposition 5.20]{KollarMori}.  Suppose $Y$ is a normal irreducible scheme with function field $L$ and that $\pi \: Y \to X$ is a finite surjective morphism.  Fix a $\bQ$-divisor $\Delta$ on $X$. There is an open set $U_{X} \subseteq X$ such that both $U_{Y} = \pi^{-1}(U_{X})$ and $U_{X}$ are regular and their respective complements each have codimension at least two.  Thus, $\Delta|_{U_{X}}$ is a $\Q$-Cartier divisor, and one can make sense of $(\pi|_{U_{Y}})^{*}(\Delta|_{U_{X}})$ on $U_{Y}$ by simply pulling back the local defining equations for each of the prime components $D_{i}$ of $\Delta$.  The $\Q$-divisor $\pi^{*}\Delta$ is then the uniquely determined by the property that $(\pi^{*}\Delta)|_{U_{Y}} = (\pi|_{U_{Y}})^{*}(\Delta|_{U_{X}})$.

Perhaps more explicitly, if $\Delta = \sum a_i D_i$, one may also define $\pi^* \Delta$ as
\[
 \pi^* \Delta := \sum_{C_j \subset Y} b_j C_j
\]
where the sum traverses all prime divisors $C_j$ on $Y$ and the coefficient $b_j$ is defined in the subsequent manner.  Set $D_j$ to be the image of $C_j$ in $X$ and let $a_j$ be the $D_j$-coefficient of $\Delta$.  Note that $C_j$ corresponds to a discrete valuation $\ord_{C_{j}} \colon L \setminus \{0\} \to \Z$ and let $t$ be a uniformizer for the discrete valuation ring (DVR) $\O_{X, D_j} \subseteq K \subseteq L$.  Then we have $b_j := a_j \cdot \ord_{C_{j}}(t)$.  The integer $\ord_{C_{j}}(t)$ is called the \emph{ramification index} of $\pi$ along~$C_{j}$.

We caution the reader that rounding operations on divisors do not in general commute with pull-back.   However, one always has
\begin{equation}
  \label{eq:2}
  \pi^{*}\lceil \Delta \rceil \geq \lceil \pi^{*} \Delta \rceil
\end{equation}
simply because, in the notation used above, $\lceil a_{j} \rceil \cdot \ord_{C_{j}}(t) \geq \lceil a_{j} \cdot \ord_{C_{j}}(t) \rceil$ as the ramification index $\ord_{C_{j}}(t)$ is an integer.  In fact, one may easily conclude the stronger statement
\begin{equation}
  \label{eq:3}
 0 \leq  \ord_{C_{j}} (\pi^{*} \lceil \Delta \rceil - \lceil \pi^{*} \Delta \rceil ) \leq \ord_{C_{j}}(t) - 1 \, \, .
\end{equation}

\subsection{Reflexive sheaves}
\label{sec:reflexivesheaves}




To any divisor $D$ on a normal irreducible scheme $X$,
there is an associated coherent subsheaf $\O_X(D)$ of the constant sheaf $K$ of rational functions on $X$.
Recall that a coherent sheaf $\sM$ on $X$ is said to be {\it reflexive} if the natural map
 $\sM \rightarrow \sM^{\vee\vee}$  to its double dual   (with respect
 to $(\blank)^{\vee} =\sHom_{\O_X}(\blank, \O_X)$) is an isomorphism.
 The sheaf $\O_X(D)$  is reflexive for any Weil divisor $D$ and is invertible if and only if $D$ is a Cartier divisor.
If $\sM$ is a rank one reflexive and has a global section $s$, then the zero set of $s$ (counting
  multiplicities) determines a divisor $D \geq 0$ on $X$ and an isomorphism $\sM \cong \O_X(D)$ sending $s$ to the element $1$ of the function field $K$.  Furthermore, every divisor $D \geq 0$ with $\sM \cong \O_X(D)$ is determined by such a section.

We state here two useful properties of reflexive modules which shall be used throughout.

\begin{itemize}
\item[(1)]  If $\pi \: Y \to X$ is a finite surjective map of normal integral schemes and $\sM$ is a coherent sheaf on $Y$, then $\sM$ is reflexive (as an $\O_Y$-module) if and only if $\pi_* \sM$ is reflexive (as an $\O_{X}$-module).  See \cite[Exercise 16.7]{MatsumuraCommutativeRingTheory} and also \cite[Corollary 1.7]{HartshorneStableReflexiveSheaves1}.
\item[(2)]  Suppose $\sM$ is reflexive coherent sheaf on a normal integral scheme $X$, and  let $i \: U \rightarrow X$ be the inclusion of an open set whose compliment has codimension at least two.  Then $i_* (\sM|_U) \cong \sM$. More generally,  restriction to $U$ induces an equivalence of categories between reflexive coherent sheaves on $X$ and reflexive coherent sheaves on~$U$ \cite[Proposition 1.11, Theorem 1.12]{HartshorneGeneralizedDivisorsOnGorensteinSchemes}.
\end{itemize}

\subsection{Canonical modules and duality}
\label{sec:CanModDuality}

In this section, we recall the canonical module and a method of associating homomorphisms arising in the study of finite maps with divisors.  References include \cite{KunzKahlerDifferentials, MehtaRamanathanFrobeniusSplittingAndCohomologyVanishing,FedderFPureRat,MehtaSrinivasFPureSurface,HaraWatanabeFRegFPure,SchwedeFAdjunction}.

 When $X$ is a normal variety of finite type over a
field, $\omega_X$ is the unique reflexive sheaf agreeing with the sheaf $\wedge^{\dim X}\Omega_X$ on the smooth locus of $X$.
More generally, we can define canonical modules on any normal integral scheme with a dualizing complex $\omega_X^{\mydot}$
\cite{HartshorneResidues}.  If $X$ happens to be Cohen-Macaulay, $\omega_X^{\mydot}$ is concentrated in one degree and we call this the \emph{dualizing sheaf}; for arbitrary normal $X$, the \emph{canonical sheaf} $\omega_X$ is the unique reflexive sheaf agreeing with the dualizing sheaf on the Cohen-Macaulay locus (whose compliment has codimension $\geq 3$).
 A \emph {canonical divisor} is any divisor $K_X$ such that $\omega_X \cong \O_X(K_X)$.

In the remainder of this section, we review an oft used duality
argument which
associates divisors to certain homomorphisms of sheaves.  The
resulting correspondences will be essential tools in our investigations.
In the simplest case, when $f \: A \to B$ is a finite inclusion of Gorenstein rings (for example, fields) with $A \cong \omega_A$ and $B \cong f^! \omega_A = \omega_B$, well-known duality
arguments give
\begin{equation}
  \label{eq:1}
  \Hom_{A}(B,A) \simeq \Hom_{A}(B,\omega_{A}) \simeq \omega_{B} \simeq B.
\end{equation}
However, this identification is not unique; a generator of
$\Hom_{A}(B,A)$ as a $B$-module is only specified up to
(pre-)multiplication by a unit in $B$.  When $B$ is normal, one may
remove this ambiguity by using the formalism of divisors.  If $\phi \in \Hom_{A}(B,A)$ is non-zero and corresponds to $b_{\phi} \in B$, the divisor $\Div(b_{\phi})$ is independent of the choice of isomorphisms above.

More generally, focusing our attention on the geometric situation, suppose $\rho \: V \to U$ is a finite surjective morphism of
normal varieties over a field  $\Lambda$.  The subsequent arguments can be generalized to many schemes not of finite type over a field, as will be remarked and used throughout the discussion below.
Further suppose that $E$ is a divisor on $U$ and $\sL$ is a line
bundle on $V$.  Using duality for a finite map we obtain the following
isomorphism of reflexive sheaves.  Strictly speaking, one
  applies duality over the locus $U' \subseteq U$ where both $U'$ and
  $V' := \rho^{-1}(U') \subseteq V$ are smooth.  Since the respective complements
  of $U'$ and $V'$ have codimension two or higher, one may extend to
  $U$ and $V$ using the properties of
  reflexive sheaves recalled in the previous section.  In particular,
 at first glance,
  the reader is invited to imagine both $U$ and $V$ are
  smooth in the arguments that follow.
\begin{eqnarray*}
 \sHom_{\O_{U}}(\rho_{*} \sL, \O_{U}(E)) & \simeq & \sHom_{\O_{U}}(\rho_{*} \sL \tensor_{\O_{U}} \omega_{U}(-E), \omega_{U}) \\
& \simeq & \rho_{*} \sHom_{\O_{V}}(\sL \tensor_{\O_{V}} \rho^{*}\omega_{U}(-E), \omega_{V}) \\
& \simeq & \rho_{*} (\omega_{V} \tensor_{\O_{V}} \rho^{*}
\omega_{U}^{-1}(E)  \tensor_{\O_{V}} \sL^{-1})\\
& \simeq & \rho_{*} (\sL^{-1}(K_{V} - \rho^* K_{U} + \rho^* E)) \, \, .
\end{eqnarray*}

Suppose now that $0 \neq \phi \in
\Hom_{\O_{U}}(\rho_{*} \sL, K(U))$ for some line bundle $\sL$.  There
exists some $E \in \mathrm{Div}(U)$ with $\phi \in \Hom_{\O_{U}}(\rho_{*} \sL, \O_U(E))$ thus corresponding to an element
$$s_{\phi, E} \in H^{0}(V, \sL^{-1}(K_{V} - \rho^* K_{U} + \rho^* E)).$$  The
divisor $\Div(s_{\phi, E})$ is an effective divisor such that $\sL(\Div(s_{\phi,E})) \cong \O_{V}(K_{V} - \rho^* K_{U} + \rho^* E)$.
\begin{definition}
With the notation as above, we define the divisor $D_{\phi}$ to be $\Div(s_{\phi, E}) - \rho^* E$.
\end{definition}
It is straightforward to verify that $D_{\phi}$ is independent of the choice of $E$ (simply view $\phi$ as a rational section of $\sL^{-1}(K_{V} - \rho^* K_{U})$).
Therefore, the association $\phi \mapsto D_{\phi}$ is independent of
the choices made in the identifications above, and gives a correspondence
\begin{equation}
\label{eq:mapstodivisorsgeneral}
\left\{ \parbox{1.4in}{
\begin{center}
Line bundles $\sL$ on $V$ and non-zero $\O_{U}$-linear maps $\phi: \rho_{*}\sL \to K(U)$
\end{center}
} \right\}
\quad \xymatrix{ \ar@{~>}[r] & }\quad
\left\{ \parbox{1.7in}{
\begin{center}
Divisors $B$ on $V$ such that $\O_V(B) \cong \sL^{-1}(K_{V} - \rho^* K_{U})$.
\end{center}
} \right\}.
\end{equation}
wherein two
maps give rise to the same divisor if and only if they differ by
pre-multiplication by an element of $H^{0}(V, \O_{V}^{*})$.
More generally, when
working with schemes which are not of finite type over a
field, the same arguments apply if we have $\omega_{V} \simeq \sHom_{\O_U}(f_* \O_V, \omega_U) =: \rho^{!}\omega_{U}$ (which holds automatically in many cases).

\begin{remark}
\label{rem:TwistOfTraceTypeMaps}
Notice that given any element of $\Hom_{\O_U}(\rho_* K(V), K(U))$, we can restrict the source to obtain a map $\rho_* \O_{V} \to K(U)$.
Of course, we can restrict to other line bundles as well.  Given any
two Cartier divisors $L_1$ and $L_2$ on $V$, we obtain maps by
restriction; $\phi_i : \rho_* \O_V(L_i) \to K(U)$ for $i = 1, 2$.
Furthermore, we have the relation $D_{\phi_1} + L_1 = D_{\phi_2} +
L_{2}$.

\end{remark}
We will typically be assuming that $\sL = \O_{V}$ and $\phi(\rho_*
\O_{V}) \subseteq \O_U$.  In this case, one may take $E = 0$ so that
$D_{\phi}$ is effective, and
(\ref{eq:mapstodivisorsgeneral}) specializes to:
\begin{equation}
\label{eq:mapstodivisors}
\left\{ \parbox{1.2in}{
\begin{center}
Non-zero $\O_{U}$-linear maps $\phi: \rho_{*}\O_{V} \to \O_{U}$
\end{center}
} \right\}
\quad \xymatrix{ \ar@{~>}[r] & }\quad
\left\{ \parbox{1.6in}{
\begin{center}
Effective divisors on $U$ linearly equivalent to $K_{V} - \rho^{*}K_{U}$
\end{center}
} \right\}.
\end{equation}

\subsection{{$p^{-e}$}-linear maps}
\label{sec:pminuselinearmaps}

\begin{definition}
 A \emph{$p^{-e}$-linear map} between $R$-modules $M$ and $N$ is an additive map $\phi \: M \to N$ such that $\phi(r^{p^e} x) = r \phi(x)$ for all $x \in M$ and $r \in R$.  A $p^{-e}$-linear map may also be viewed as an $R$-linear map $F^e_* M \rightarrow N$.  We will therefore abuse terminology and refer to the elements of $\Hom_R(F^e_* M, N)$ as $p^{-e}$-linear maps.  Likewise, if $M$ is a submodule of the fraction field of $R$, we will also say that an element of $\Hom_R(M^{1/p^e}, N)$ is a $p^{-e}$-linear map.
\end{definition}

The purpose of this section is to apply the duality arguments above (with a slight modification) to associate a $\bQ$-divisor to each non-zero $p^{-e}$-linear map.  To that end,
suppose that $X$ is an $F$-finite normal scheme in characteristic $p > 0$ with a canonical module $\omega_X$ (and not necessarily essentially of finite type over a field).  Consider the following condition.
\begin{equation}  \label{eqn:FUpperShriekOmegaIsOmega} \tag{$!$}
F^! \omega_X = \Hom_{\O_X}(F_* \O_X, \omega_X) \simeq \omega_X.
\end{equation}
Condition (\ref{eqn:FUpperShriekOmegaIsOmega}) holds if $X$ is the spectrum of a local ring or essentially of finite type over an $F$-finite field (\ie a variety).
Furthermore, (\ref{eqn:FUpperShriekOmegaIsOmega}) always holds on sufficiently small affine charts (one always has $F^! \omega_X \simeq \omega_X \tensor \sN$ for some line bundle $\sN$ \cite[Chapter V, Theorem 3.1]{HartshorneResidues}).

\bigskip

\noindent
  \begin{tabular}[c]{ll}
    {\scshape Convention:} & \parbox[t]{4.5in}{All schemes and rings in characteristic $p > 0$ we consider in this paper will be assumed to satisfy condition (\ref{eqn:FUpperShriekOmegaIsOmega}).}
  \end{tabular}

\begin{remark}
 The assumption above is made essentially for convenience and readability.  It is possible, but rather messy, to keep track of the line bundle $\sN$ such that $F^! \omega_X \simeq \omega_X \tensor \sN$.
\end{remark}

If $X$ is a normal $F$-finite scheme satisfying condition (\ref{eqn:FUpperShriekOmegaIsOmega}), applying a variant of (\ref{eq:mapstodivisors}) yields the following bijective correspondence; see \cite{SchwedeFAdjunction}, and compare with \cite{HaraWatanabeFRegFPure,MehtaRamanathanFrobeniusSplittingAndCohomologyVanishing}.
\begin{equation}
\label{eq:divmapcorr}
\left\{ \parbox{1.7in}{
\begin{center}
Non-zero $\O_{X}$-linear maps $\phi: F^{e}_{*}\O_{X} \to \O_{X}$ up to pre-multiplication by an
element of $H^{0}(X, \O_{X}^{*})$
\end{center}
} \right\}
\quad \longleftrightarrow \quad
\left\{ \parbox{1.3in}{
\begin{center}
Effective $\Q$-divisors $\Delta$ on $X$ such that $(1-p^{e})(K_{X}+\Delta)$ is principal
\end{center}
} \right\}.
\end{equation}

Recall that a principal divisor is simply the divisor of a rational
function (it is necessarily an integral divisor).
Given $\phi \in \Hom_{\O_{X}}(F^{e}_{*}\O_{X}, \O_{X})$, one constructs the associated $\Q$-divisor $\Delta_{\phi}$
as follows:  for every prime Weil divisor $E$ on $X$ with $\O_{X,E}$ the stalk at the generic point of $E$, we
have
\[
\Hom_{\O_{X,E}}(F^{e}_{*}\O_{X,E},\O_{X,E}) =
\Hom_{\O_{X,E}}(\O_{X,E}^{1/p^e},\O_{X,E}) \simeq \O_{X,E}^{1/p^e}
\]
as in $\eqref{eq:1}$, since $\O_{X,E}$ is an $F$-finite discrete valuation ring
(DVR).  Suppose $\Phi_{E}$ generates
$\Hom_{\O_{X,E}}(\O_{X,E}^{1/p^e},\O_{X,E})$ as an
$\O_{X,E}^{1/p^e}$-module and write $\phi_{E}(\blank) = \Phi_{E}(
f^{1/p^e} \cdot \blank)$ where $f \in \O_{X,E}$ (\ie $\phi_{E}$ is the composition of $\Phi_{E}$ with multiplication by $f^{1/p^e}$).  The coefficient of $E$ in $\Delta_{\phi}$ is then
\begin{equation*}
  \ord_{E}(\Delta) := \frac{1}{p^e-1} \ord_{E}(f).
\end{equation*}
Rescaling by $\frac{1}{p^e-1}$ makes the above correspondence
\eqref{eq:divmapcorr} stable under iteration via Frobenius.  In other words, for all $n \geq 0$, the $\Q$-divisor $\Delta_{\phi^{n}}$ associated to
\begin{equation}
\label{eqn:ComposeMapWithSelf}
 \phi^{n} := \phi \circ F^{e}_{*} \phi \circ F^{2e}_{*} \circ \cdots \circ F^{(n-1)e}_{*} \phi \in \Hom_{\O_{X}}(F^{ne}_{*} \O_{X},\O_{X})
\end{equation}
is also $\Delta_{\phi}$.  See \cite{SchwedeFAdjunction} for further details.

One may generalize (\ref{eq:divmapcorr}) to include all
$\Q$-divisors $\Delta \geq 0$ on $X$ such that \mbox{$(p^{e}-1)(K_{X}+
\Delta)$} is simply Cartier, \ie such that $(1-p^{e})(K_{X}+\Delta)$
is integral and the associated reflexive sheaf $\sL =\O_{X}((1-p^{e})(K_{X}+\Delta))$ is invertible.  Explicitly, define an equivalence relation $\star$ on invertible
sheaves $\sL$ with a non-zero $\O_{X}$-linear map $\phi \: F^{e}_{*}\sL
\to \O_{X}$ as follows:  write
\begin{equation*}
  (\sL_{1}, \phi_{1} \: F^{e}_{*} \sL_{1} \to \O_{X})  \quad
  \star  \quad
  (\sL_{2}, \phi_{2} \: F^{e}_{*} \sL_{2} \to \O_{X})
\end{equation*}
if there is an $\O_X$-module isomorphism $\gamma \: \sL_{1} \to
\sL_{2}$ fitting in a commutative diagram
\[  \xymatrix{
F^{e}_{*}\sL_{1}    \ar_{\phi_{1}}[dr]    \ar^{F^{e}_{*}\gamma}[rr]    &   &  F^{e}_{*}\sL_{2}   \ar^{\phi_{2}}[dl]   \\
& \O_{X} & \\
}. \]
The bijection in \eqref{eq:divmapcorr} now takes the form below (see
\cite{SchwedeFAdjunction,BlickleSchwedeTakagiZhang}).

\begin{equation}
\label{eq:divmaplbcorr}
\left\{ \parbox{1.3in}{
\begin{center}
Invertible sheaves $\sL$ with a non-zero $\O_{X}$-linear map $\phi:
F^{e}_{*}\sL \to \O_{X}$ up to  $\star$-equivalence
\end{center}
} \right\}
\quad \longleftrightarrow \quad
\left\{ \parbox{1.8in}{
\begin{center}
Effective $\Q$-divisors $\Delta$ on $X$ such that
$(1-p^{e})(K_{X}+\Delta)$ is Cartier with
$\O_{X}((1-p^{e})(K_{X}+\Delta)) \simeq \sL$
\end{center}
} \right\}.
\end{equation}

\section{Extending finite maps over finite separable extensions}
\label{sec:LiftingProperties}


Suppose $R \subseteq S$ is a module-finite inclusion of $F$-finite
normal domains of characteristic $p>0$ and that the corresponding
extension of fraction fields $K := \Frac(R) \subseteq L := \Frac(S)$
is separable (the general case will be treated in subsequent sections).  Assume further that $R$ and $S$ satisfy (\ref{eqn:FUpperShriekOmegaIsOmega}); at the risk of needless repetition, we remark that property (\ref{eqn:FUpperShriekOmegaIsOmega}) is automatically satisfied when $R$ and $S$ are semi-local or are essentially of finite type over a perfect or even $F$-finite field.

We will give a criterion for when a $p^{-e}$-linear
map $\phi
\in \Hom_{R}(R^{1/p^e},R)$ on $R$ extends to a $p^{-e}$-linear
map \mbox{$\bar{\phi}
\in \Hom_{S}(S^{1/p^e}, S)$} on $S$, \ie such that the following diagram
commutes:
\begin{center}
\begin{minipage}[b]{\linewidth}
\[  \xymatrix{
S^{1/p^e} \ar^{\bar{\phi}}[r] & S
}\]

\vspace{-.5cm}
\begin{equation}
\label{eq:extndiag}
 \xymatrix{
 \rotatebox[origin=c]{90}{$\subseteq$} & \quad \rotatebox[origin=c]{90}{$\subseteq$}
 }\end{equation}

\vspace{-.5cm}
\[ \xymatrix{
R^{1/p^e} \ar_{\phi}[r] & R
}. \]
\end{minipage}
\end{center}
This criterion will be given in terms of the associated divisor
$\Delta_{\phi}$ on $\Spec R$ as in \eqref{eq:divmapcorr}.  
First, however, consider the following heuristic argument.

\begin{remark}
\label{rem:TooEasyArgument}
Let $X = \Spec R$ and $Y = \Spec S$ with induced map $\pi \: Y \to X$.
We have seen that $\phi \in \Hom_{R}(R^{1/p^e}, R)$ corresponds to a
section of $\O_X( (1-p^e)K_X)$.  Assuming $\pi$ is separable, we have $\pi^* ((1-p^e)K_X) = (1-p^e)(K_Y - \Ram_{\pi})$ where $\Ram_{\pi}$ is the ramification divisor of $\pi$ (see Definition~\ref{def:ramdiv}).  Therefore, after pulling back, $\phi$ induces a section of $\O_Y((1-p^e)(K_Y - \Ram_{\pi}))$ and hence also a map
\[
\bar{\phi} \in \Hom_{\O_Y}\left(\Frobp{\O_Y((1-p^e)\Ram_{\pi})}{e}, \O_Y\right).
\]
Thus we can view $\bar{\phi}$ as a rational section of $\Hom_S(F^e_*
S, S)$ and it is natural to hope that $\bar{\phi}$ extends $\phi$.
However, $\bar{\phi}$ as described is not canonically chosen;
unwinding the isomorphisms above gives a $p^{-e}$-linear map only up to multiplication by a unit of
$S$. As a result, it is difficult to see whether or not $\bar{\phi}$ extends $\phi$.
\end{remark}


In Theorem \ref{thm:mainliftingcriterion} we essentially show
how to fill the gap in
the line of argument of Remark~\ref{rem:TooEasyArgument} by using a canonical global section of
$\Hom_X(\pi_* \O_Y, \O_X) \cong \pi_* \O_Y(\Ram_{\pi})$, namely the
\emph{trace map}.  Furthermore, we will see more generally that fixing a homomorphism $\Tt \in
\Hom_X(\pi_* \O_Y, \O_X)$ allows one to handle similar issues for
arbitrary (\ie not necessarily separable) finite surjective morphisms
of normal varieties.

\begin{remark}
In the absence of wild ramification in codimension one, it is possible
to compute extensions rather directly (choosing bases appropriately)
and thereby
avoid the difficulties in Remark \ref{rem:TooEasyArgument} above.  Although more restrictive, we
have found this approach to be highly instructive.  This description will appear separately in \cite{SchwedeTuckerExplicitLiftingOfFSplittings}.
\end{remark}

\subsection{Generic and \etale extensions}
\label{sec:genericextensions}

\begin{lemma}
\label{lem:genericextension}
Suppose that $K \subseteq L$ is a finite separable field extension
  of $F$-finite fields with characteristic $p>0$.  If
\[
 e_{1}^{1/p^e}, \ldots, e_{n}^{1/p^e}
 \]
 form a basis for $K^{1/p^e}$ over $K$, then they also form a basis for $L^{1/p^e}$ over $L$. Furthermore, any
  map $\phi \in \Hom_{K}(K^{1/p^e},K)$ extends uniquely
  to a map $\bar{\phi} \in \Hom_{L}(L^{1/p^e}, L)$.
\end{lemma}

\begin{proof}
 Since $K^{1/p^e}$ is purely inseparable over $K$, it follows that
 $K^{1/p^e}$ and $L$ are linearly disjoint over $K$ (\eg Example 20.13
 in \cite{Mor96}).  Thus, $L^{1/p^e} = K^{1/p^e} \tensor_{K}
 L$, and the first statement is obvious.  For the second, note that
 the extension $\bar\phi = \phi \tensor_{K} \Id_{L}$ exists and is
 uniquely determined by the property that $\bar\phi(e_{i}^{1/p^e}) =
 \phi(e_{i}^{1/p^e})$ for all $1 \leq i \leq n$.
\end{proof}

\begin{example}
\label{ex:y=x2}
  Suppose $R = \F_{3}[y] \subseteq \F_{3}[x,y]/(y - x^{2}) \simeq
  \F_{3}[x] = S$.  Then $R^{1/3}$ is a free \mbox{$R$-module} with basis $1,
  y^{1/3}, y^{2/3}$, and similarly for $S^{1/3}$.  Thus, $\phi \in
  \Hom_{R}(R^{1/3},R)$ is uniquely determined by $\phi(1),
  \phi(y^{1/3}), \phi(y^{2/3})$, which can be arbitrary elements of
  $R$.  We identify $\phi$ with its generic extension $\phi \: K^{1/3}
  \to K$ to $K = \Frac(R)$, and denote by $\bar\phi$ the unique extension of $\phi$ to
  $L = \Frac(S)$ as in Lemma \ref{lem:genericextension}.  We have
  \begin{equation*}
    \bar\phi(1) = \phi(1) \qquad \bar\phi(x^{1/3}) = \frac{1}{x}
    \phi(y^{2/3}) \qquad \bar\phi(x^{2/3}) = \phi(y^{1/3})
  \end{equation*}
and it follows that an extension of $\phi$ from $R$ to $S$ as in
\eqref{eq:extndiag} exists if and only if $\phi(y^{2/3})$ is divisible
by $x$ in $S$.  Furthermore, this condition holds if
and only if $y$ divides $\phi(y^{2/3})$ in $R$. Note that this can be
checked after localizing at $\langle y \rangle$;  as we shall see, it is not a coincidence that $x$
defines the ramification locus in this example.
\end{example}

\begin{lemma}
\label{lem:etaleextens}
  Suppose $R \subseteq S$ is a module-finite \etale inclusion of domains in
  characteristic $p > 0$.  Then a map $\phi \in
  \Hom_{R}(R^{1/p^e}, R)$ extends uniquely to a map
  $\bar\phi \in \Hom_{S}(S^{1/p^e},S)$.
\end{lemma}

\begin{proof}
  Though this fact is well-known, we sketch a proof for
  completeness.  As in the proof of Lemma \ref{lem:genericextension},
  it suffices to show (following the argument given in
  \cite[Section 6.3]{HochsterHunekeTC1}) that the natural map
$R^{1/p^e}\tensor_{R} S \to S^{1/p^e}$
is an isomorphism.  This can be verified locally on $R$, and hence we
may assume $(R, \m, k)$ and thus also $(R^{1/p^e}, \m^{1/p^e}, k^{1/p^e})$ are local.
Both  $R^{1/p^e}\tensor_{R} S$ and  $S^{1/p^e}$ are flat and therefore free
$R^{1/p^e}$-modules, and by Nakayama's lemma we can check the
isomorphism after killing $\m^{1/p^e}$.  Since $S^{1/p^e}/(\m^{1/p^e}S^{1/p^e})$
is a product of finite separable field extensions of $k^{1/p^e}$, the
statement follows from Lemma \ref{lem:genericextension}.
\end{proof}

\subsection{Codimension one}
\label{sec:reduct-codim-one}

In this section, we use localization to reduce our extension question to
codimension one, \ie to the case of a discrete valuation ring (DVR).

\begin{lemma}
\label{lem:reductcodimoneaffine}
  Suppose $R \subseteq S$ is a generically separable module-finite inclusion
  of $F$-finite normal domains with characteristic $p>0$ and $\phi \in
  \Hom_{R}(R^{1/p^e},R)$.  Then $\phi$ extends to
  $\bar\phi \in \Hom_{S}(S^{1/p^e},S)$ if and only if an extension exists
  in codimension one.  In other words, for each height one prime $\bq$
  of $S$ lying over $\bp$ in $R$, we have a commutative diagram
\begin{center}
\begin{minipage}{\linewidth}
\[  \xymatrix{
(S_{\bq})^{1/p^e} \ar^-{\bar{\phi}}[r] & S_{\bq}
}\]

\vspace{-.5cm}
\begin{equation*}
 \xymatrix{
 \rotatebox[origin=c]{90}{$\subseteq$} & \quad \quad \rotatebox[origin=c]{90}{$\subseteq$}
 }\end{equation*}

\vspace{-.5cm}
\[ \xymatrix{
(R_{\bp})^{1/p^e} \ar_-{\phi}[r] & R_{\bp}
}. \]
\end{minipage}
\end{center}
\end{lemma}

\begin{proof}
Identify $\phi$ with its generic extension $\phi \: K^{1/p^e} \to   K$ to $K = \Frac(R)$, and denote by $\bar\phi$ the unique extension  of $\phi$ to $L = \Frac(S)$ as in Lemma \ref{lem:genericextension}.  Then, an extension of $\phi$ to $S$ exists if and only if  $\bar\phi(S^{1/p^e}) \subseteq S$.  Since $S$ is normal,  $S$ is the intersection of all of the subrings $S_{\bq} \subseteq L$  for each height one prime $\bq$ of $S$, and the conclusion follows at once.
\end{proof}

\begin{corollary}
\label{cor:lbcodim1}
 Suppose $\pi \: Y \to X$ is a finite separable morphism
of irreducible $F$-finite normal schemes
and $\phi \in \Hom_{\O_{X}}(F^{e}_{*}\sL, \O_{X})$ where $\sL$ is an
invertible $\O_{X}$-module. Then $\phi$ has
an extension $\bar\phi \in \Hom_{\O_{Y}}(F^{e}_{*}\pi^{*} \sL,
\O_{Y})$, \ie such that the diagram
\[  \xymatrix{
    \pi_{*} \, F^{e}_{*} \pi^{*} \sL \ar^-{\pi_{*}\bar\phi}[r] & \pi_{*}\O_{Y}
    \\
    F^{e}_{*}\sL \ar^{\rotatebox[origin=c]{90}{$\subseteq$}}_{\iota}[u] \ar_{\phi}[r] & \O_{X} \ar^{\pi^{\sharp}}_{\rotatebox[origin=c]{90}{$\subseteq$}}[u]
} .\]
commutes, if
and only if an extension exists in codimension one.  Here,
$\pi^{\sharp}$ is the natural map, and the
vertical map on the left is simply $\iota = F^{e}_{*}(\pi^{\sharp} \tensor
\id_{\sL})$ after the identifications
\[
\pi_{*}F^{e}_{*}\pi^{*} \sL = F^{e}_{*}\pi_{*}\pi^{*} \sL = F^{e}_{*}
(\pi_{*}\O_{Y} \tensor_{\O_{X}} \sL).
\]
\end{corollary}

\begin{proof}
  These statements follow immediately from  Lemmas
  \ref{lem:reductcodimoneaffine}
  and
  \ref{lem:etaleextens}.
\end{proof}

\begin{example}
  Suppose $R = \Lambda[x^{3},x^{2}y,xy^{2},y^{3}] \subseteq
  \Lambda[x,y] = S$, where $\Lambda$ is any $F$-finite field with
  characteristic $p \not\in \{0,3\}$.  Since this extension is \etale in codimension
  one, every map $\phi \in \Hom_{R}(R^{1/p^e},R)$ extends
  to a map $\bar\phi \in \Hom_{S}(S^{1/p^e},S)$.
\end{example}

Thus, we see that the extension question for $p^{-e}$-linear
maps reduces to understanding what happens in the presence of
ramification in codimension one.
In order to highlight the situation for various kinds of
ramification, let us first recall what it means for an inclusion of DVR's to be tamely ramified.

\begin{definition}  \cite[Chapter 2]{GrothendieckMurreTheTameFundamentalGroup} \label{defn:tameramification}
 A local inclusion $(R, \bp) \subseteq (S, \bq)$ of DVR's is called \emph{tamely ramified} if:
\begin{itemize}
 \item[(1)]  It is generically finite and generically separable.
 \item[(2)]  The extension of residue fields $k(\bp) \subseteq k(\bq)$ is separable.
 \item[(3)]  A local parameter $r$ of $R_{\bp}$, when viewed as an element of $S_{\bq}$, has order of vanishing (with respect to the valuation of $S_{\bq}$) not divisible by $p$.
\end{itemize}
More generally, a generically separable module-finite inclusion of
normal domains $R \subseteq S$ will be called \emph{tamely ramified in
  codimension one} if all of the associated DVR extensions given by
localizing at height one primes are tamely ramified.  A generically
finite and separable inclusion of DVR's which is not tamely ramified is called \emph{wildly ramified}.
\end{definition}

In Example \ref{ex:y=x2}, we saw a tamely ramified extension and what
conditions were necessary for a $p^{-e}$-linear map to extend.  In
contrast,
we consider an example of a wildly ramified extensions.  
\begin{example}[Wild ramification via ramification index]
\label{ex:nottame1}
Suppose $ R = \mbox{$\F_{2}[[x^2(1+x^3)]]$} = \F_{2}[[t]]$ and $S =
\F_{2}[[x]]$ with the natural inclusion $R \subseteq S$. (This
example occurs geometrically if one completes $\F_2[x^2(1+x^3)] =\F_{2}[t]
\subseteq \F_2[x] $ at the origin $\langle t \rangle$ of the smaller
ring; the completion of $\F_2[x]$ along $\langle t \rangle$ splits
into a direct product of rings, with $S$ appearing as a factor.)


Consider the maps $\phi \in \Hom_{R}(R^{1/2},R)$ and its extension $\bar{\phi} \in \Hom_{S}(S^{1/2},S)$.
\begin{equation*}
  \begin{array}[t]{rcl}
    \phi \: R^{1/2} & \to & R \\
    1 & \mapsto & 0 \\
    t^{1/2} & \mapsto & t
  \end{array}
\qquad
  \begin{array}[t]{rcl}
    \bar\phi \: S^{1/2} & \to & S \\
    1 & \mapsto & 0 \\
    x^{1/2} & \mapsto & (1+x^3)
  \end{array}.
\end{equation*}
It is easy to verify that $\bar\phi$ extends $\phi$.  Indeed, we have $\bar{\phi}(1) = \phi(1) = 0$, and to verify the assertion that we must have $\bar{\phi}(x^{1/2}) = 1 + x^3$ simply divide both sides of the following by $x^2$,
\[
 t = x^2(1+x^3) = \bar{\phi}(x^{2/2}+x^{5/2}) = \bar{\phi}(x^{2/2}) + \bar{\phi}(x^{5/2}) = x \bar\phi(1) + x^2 \bar\phi(x^{1/2}) =  0 + x^2 \bar{\phi}(x^{1/2}).
\]

In fact, the map $\bar{\phi}$ is an $S^{1/2}$-module generator of $\Hom_S(S^{1/2}, S)$.
We will see in Section \ref{sec:tracemap} that $\psi \: R^{1/2} \rightarrow R$ extends to an $S$-linear map $\bar{\psi} \: S^{1/2} \to S$ if and only if $\Delta_{\psi} \geq \Delta_{\phi}$.  In other words, only the $R^{1/2}$-multiples of $\phi$ extend to $S$.
In this example, a non-surjective map $\phi \: R^{1/2} \to R$ extends to a surjective map $\bar{\phi} \: S^{1/2} \to S$.

\end{example}

\section{The trace map and $p^{-e}$-linear maps}
\label{sec:tracemap}


\subsection{Commutativity of trace with extensions of $p^{-e}$-linear maps}
\label{sec:commtrace}

\begin{proposition}
\label{prop:fieldtracecommute}
  Suppose that $K \subseteq L$ is a finite separable field extension
  of $F$-finite fields of characteristic $p>0$.  Let $\phi \in
  \Hom_{K}(K^{1/p^e},K)$, and denote by
  $\bar\phi \in \Hom_{L}(L^{1/p^e},L)$ its unique extension.
  If $\Tr_{L/K} \: L \to K$ is the trace map of $L$ over $K$, the
  following diagram commutes
\[  \xymatrix{
K^{1/p^e}       \ar^{\phi}[r]    &     K   \\
L^{1/p^e}    \ar_{\bar\phi}[r]     \ar^{(\Tr_{L/K})^{1/p^e}}[u]            &     L   \ar_{\Tr_{L/K}}[u]
} .\]
\end{proposition}

\begin{proof}
  Note that $(\Tr_{L/K})^{1/p^e} = \Tr_{L^{1/p^e}/K^{1/p^e}}$ is
  simply the trace map of $L^{1/p^e}$ over $K^{1/p^e}$.  We claim that $\Tr_{L^{1/p^e}/K^{1/p^e}}\vert_{L} =
  \Tr_{L/K}$.  Indeed, choose a basis
  $f_{1}, \ldots, f_{m}$ for $L$ over $K$. Since $K^{1/p^e}$ and $L$ are
  linearly disjoint over $K$, these are also a basis for $L^{1/p^e}$ over
  $K^{1/p^e}$(\cf Lemma \ref{lem:genericextension}).  Thus, if $\lambda \in
  L$, we have identical matrix expressions (with respect to $f_{1}, \ldots, f_{m}$)
  for multiplication by $\lambda$ on both $L$ as a
  vector space over $K$ and also on $L^{1/p^e}$ as a vector space over
  $K^{1/p^e}$, and our claim follows.

Suppose $\mu^{1/p^e} \in L^{1/p^e}$, and write $\mu^{1/p^e} = \sum_{i=1}^{m}
e_{i}^{1/p^e}f_{i}$ where $e_{1}^{1/p^e}, \ldots, e_{m}^{1/p^e} \in
K^{1/p^e}$ and $f_1, \dots, f_m \in L$.  Explicit calculation gives
\begin{eqnarray*}
 \left( \phi \circ (\Tr_{L/K})^{1/p^e} \right)\left( \mu^{1/p^e} \right) &=& \phi\left(
    \Tr_{L^{1/p^e}/K^{1/p^e}}(\sum_{i=1}^{m}e_{i}^{1/p^e}f_{i})\right) \\
&=& \phi\left(
  \sum_{i=1}^{m}e_{i}^{1/p^e}\Tr_{L^{1/p^e}/K^{1/p^e}}(f_{i})\right) \\
&=& \sum_{i=1}^{m} \phi(e_{i}^{1/p^e}) \Tr_{L/K}(f_{i}).
\end{eqnarray*}
Similarly,
\begin{eqnarray*}
  \left( \, \Tr_{L/K} \circ \, \bar\phi^{\phantom{.}} \right) \left( \mu^{1/p^e}
  \right) &=& \Tr_{L/K} \left( \bar\phi(\sum_{i=1}^{m}e_{i}^{1/p^e}f_{i})
\right) \\
& =& \Tr_{L/K}\left(  \sum_{i=1}^{m}  \phi(e_{i}^{1/p^e})f_{i}\right)\\
& = & \sum_{i=1}^{m} \phi(e^{1/p^e}) \Tr_{L/K}(f_{i}),
\end{eqnarray*}
and the conclusion follows.
\end{proof}

\begin{corollary}
\label{cor:trcommutes}
   Suppose $\pi \: Y \to X$ is a finite separable morphism
of irreducible $F$-finite normal schemes
and $\sL$ is an
invertible sheaf on $X$.  If
$\phi \in \Hom_{\O_{X}}(F^{e}_{*}\sL, \O_{X})$
has
an extension $\bar\phi \in \Hom_{\O_{Y}}(F^{e}_{*}\pi^{*} \sL,
\O_{Y})$,  then the diagram
\[  \xymatrix{
F^{e}_{*}\sL       \ar^{\phi}[r]    &     \O_{X}   \\
\pi_{*}F^{e}_{*} \pi^{*} \sL    \ar_{\pi_{*}\bar\phi}[r]
\ar^{F^{e}_{*}(\Tr_{Y/X} \tensor \id_{\sL})}[u]            &     \pi_{*}\O_{Y}   \ar_{\Tr_{Y/X}}[u]
} .\]
commutes.  In particular, if $R \subseteq S$ is a generically separable module-finite
inclusion of $F$-finite normal domains and $\phi \in \Hom_{R}(R^{1/p^e},R)$
extends to $\bar\phi \in
\Hom_{S}(S^{1/p^e},S)$, then we have a commutative diagram
\begin{equation}
\label{eq:RScommTr}
\parbox[c]{1.5in}{$
  \xymatrix{
R^{1/p^e}       \ar^{\phi}[r]    &     R   \\
S^{1/p^e}    \ar_{\bar\phi}[r]     \ar^{(\Tr_{L/K})^{1/p^e}}[u]            &     S   \ar_{\Tr_{L/K}}[u]
} .$}\end{equation}
\end{corollary}

\begin{proof}
  By considering the relevant sheaves as subsheaves of the function
  fields $K(Y)$ and $K(X)$ of $Y$ and $X$, respectively, the statement
  follows from Proposition \ref{prop:fieldtracecommute}.
\end{proof}



\begin{example}
\label{ex:nocommute}
The commutativity in \eqref{eq:RScommTr} does not hold in
  general for arbitrary maps (or even splittings) in
    $\Hom_{R}(S,R)$, as we now demonstrate.
In the situation of Example \ref{ex:y=x2} ($R = k[x^2] = k[y] \subseteq k[x] = S$), we have that $S$ is a free $R$-module
with basis $1,x$ and can define $\tau
\in \Hom_{R}(S,R)$ by $\tau(1) = 1$ and $\tau(x) = y^{2}$.   Even
though $\tau$ splits the inclusion of $R$ in $S$, it does not satisfy
the commutativity property above.  In fact,
consider the
$p^{-e}$-linear map $\phi$ and extension $\bar\phi$ given by
\begin{equation*}
  \begin{array}[t]{rcl}
    \phi \: R^{1/3} & \to & R \\
    1 & \mapsto & 1 \\
    y^{1/3} & \mapsto & 1 \\
    y^{2/3} & \mapsto & y
  \end{array}
\qquad
  \begin{array}[t]{rcl}
    \bar\phi \: S^{1/3} & \to & S \\
    1 & \mapsto & 1 \\
    x^{1/3} & \mapsto & x \\
    x^{2/3} & \mapsto & 1
  \end{array}.
\end{equation*}
Then $\phi \, \circ \, \tau^{1/3} \neq \tau \, \circ \, \bar\phi$, since
$
  \phi\left(\tau^{1/3}(x^{1/3})\right) = \phi(y^{2/3}) = y
$ but $
  \tau\left( \bar\phi( x^{1/3}) \right) = \tau(x) = y^{2}$.
\end{example}

We next observe that the trace map is essentially the only map fitting into
the commutative
diagrams \eqref{eq:RScommTr} above.

\begin{proposition}
\label{prop.TraceIsOnlyMap}
 Let $R \subseteq S$ be a module-finite generically separable inclusion of
 $F$-finite normal domains with $K = \Frac R$ and $L = \Frac S$. Suppose that $\bar \phi \in
 \Hom_S(S^{1/p^e}, S)$ extends a non-zero map $\phi \in \Hom_R(R^{1/p^e},
 R)$, and also that $\psi \in \Hom_R(S, R)$ is such that the diagram
\[
\xymatrix{
R^{1/p^e}       \ar^{\phi}[r]    &     R   \\
S^{1/p^e}    \ar_{\bar\phi}[r]     \ar^{(\psi)^{1/p^e}}[u]            &     S   \ar_{\psi}[u]
} \]
commutes.
Then there exists an element $u \in \bF_{p^e} \cap L$ such that $\Tr_{L/K}(u \cdot \blank) = \psi(\blank)$.
\end{proposition}
\begin{proof}
 We have the following
commutative diagram (abusing notation slightly by identifying
$p^{-e}$-linear maps with their generic extensions)
\[
\xymatrix{
K^{1/p^e}       \ar^{\phi}[r]    &     K   \\
L^{1/p^e}    \ar_{\bar\phi}[r]     \ar^{(\psi)^{1/p^e}}[u]            &     L   \ar_{\psi}[u]
} .\]
Since $\Hom_K(L, K)$ is a one dimensional vector space over $L$,
we see that $\Tr_{L/K}(u \blank) = \psi(\blank)$ for some element $u \in L$ .
Thus, it follows that
\[
\Tr_{L/K}(\bar{\phi}(u \cdot \blank)) = \Tr_{L/K}(u \cdot \bar{\phi}(\blank)) = \psi (\bar{\phi}(\blank)) = \phi(\psi^{1/p^e}(\blank)) = \phi(\Tr_{L/K}^{1/p^e}(u^{1/p^e} \cdot \blank)) .
\]
But $\Tr_{L/K} \, \circ \, \bar{\phi} = \phi \, \circ \, \Tr_{L/K}$, hence
 $u = u^{1/p^e}$ or equivalently $u^{p^e} - u = 0$ (again as $\Hom_K(L, K)
 $ is one dimensional over $L$).
In particular, this means that $u$ is a member of the finite field $\bF_{p^e} \cap L$.
\end{proof}

\subsection{Trace and the ramification divisor}
\label{sec:traceramifdiv}

\begin{definition}
\label{def:ramdiv}
  Suppose $\pi \: Y \to X$ is a finite separable morphism of
  normal schemes.
  The relative sheaf of differentials $\Omega_{Y/X}$ is a torsion
  sheaf on $Y$, and recall that the \emph{ramification divisor}
  $\Ram_{\pi}$ is the effective divisor defined by the property that
\[
\ord_{E}(\Ram_{\pi}) = \length_{\O_{Y,E}}(\Omega_{Y/X,E})
\]
for all integral subschemes $E$ on $Y$ of codimension one.  Here as above,
$\O_{Y, E}$ and $\Omega_{Y/X, E}$ denote the stalks of the relevant sheaves at the generic point of $E$.
\end{definition}

\begin{remark}
\label{rmk:ramindexminusonelessthancoeffinramdiv}
  An easy computation shows that $\left( \length_{\O_{Y,E}}(\Omega_{Y/X,E}) + 1 \right)$ is at least as large as the ramification index of $\pi$ along $E$
\cite[Chapter IV, Proposition 2.2]{Hartshorne}.
\end{remark}

Our next goal is to show that the trace map corresponds to the
ramification divisor via (\ref{eq:mapstodivisors}).  This fact is
known to experts, but we are unaware of a reference in full generality.  We first need the following observation.

\begin{lemma}
\label{lem:CharacterizationOfGeneratingMap}
Let $R \subseteq S$ be a module-finite inclusion of domains with
corresponding extension of fraction fields $K \subseteq L$.
Suppose that $\Phi \in \Hom_{K}(L,K)$ satisfies $\Phi(S) \subseteq R$, and
that $\Phi|_{S} \: S \to R$ generates $\Hom_R(S, R)$ as an $S$-module.
If $x \in L$ is such that $\Phi(x S) \subseteq R$, then $x \in S$.
\end{lemma}
\begin{proof}
The map $\phi(\blank) := \Phi(x \cdot \blank)$ can be viewed as an
element of $\Hom_R(S, R)$.  Therefore $\phi(a) = \Phi(s a)$ for some
$s$ in $S$ and all $a \in S$.  But then $\Phi(s l) = \Phi(x l)$ for
all $l \in L$, which implies that $s = x$ as $\Hom_K(L, K)$ is a
one dimensional vector space over $L$.
\end{proof}

\begin{proposition}
  \label{prop:traceram}
  If $\pi \: Y \to X$ is a finite separable morphism of
  normal irreducible schemes with $\pi^{!} \omega_{X} \simeq \omega_{Y}$, then the trace
  map $\Tr_{Y/X} \: \pi_{*} \O_{Y} \to \O_{X}$ corresponds via (\ref{eq:mapstodivisors}) to the ramification
  divisor $\Ram_{\pi}$.
\end{proposition}

\begin{proof}
  The assertion can be verified by checking in codimension one on $Y$,
  and thus we may assume $X = \Spec R$ for $R$ a DVR, $Y = \Spec S$
  for $S$ a Dedekind domain, and $R \subseteq S$ a module-finite
  extension.
Also, if $K = \Frac(R)$ and $\Frac(S) = L$, we have that  $K \subseteq
L$ is a finite separable field extension and $\Tr_{Y/X} = \Tr_{L/K}$.
Since $S$ is a semi-local Dedekind domain, it is
  a principal ideal domain (PID).  In particular, it
  follows from Section \ref{sec:CanModDuality} that $\Hom_{R}(S,R)
  \simeq S$.  Choose a generator $\Phi$ for $\Hom_{R}(S,R)$ as an
  $S$-module, and write $\Tr_{L/K} = \Phi(s \cdot \blank)$ for some $s
  \in S$.  Thus, the divisor corresponding to $\Tr_{Y/X}$ via
  (\ref{eq:mapstodivisors}) is $\Div(s)$.

 We first show that the ideal $\langle s \rangle \subseteq S$
 is precisely the different ideal $\sD_{S/R}$ (see \cite[Chapter III]{SerreLocalFields}).  Indeed, identify $\Phi$ with its natural extension to
 $\Hom_{K}(L,K) \simeq L$.
For $x \in L$, it follows that $\Phi(xS) \subseteq R$ if and only if
$x \in S$ by Lemma \ref{lem:CharacterizationOfGeneratingMap}.  Thus, for $y \in L$, we have $\Tr_{L/K}(yS) \subseteq R$
if and only if $ys \in S$, or equivalently $y \in \frac{1}{s}S$.  As this is the
defining property of the inverse different, we conclude $\frac{1}{s}S =
\sD_{S/R}^{-1}$ or $\sD_{S/R} = \langle s \rangle$ as desired.

By the structure theorem for modules over a PID,
we have that $\Omega_{S/R} = \bigoplus_{i=1}^{m} S/a_{i}S$ for some $a_{1},
\ldots, a_{m} \in S\setminus\{0\}$.
Furthermore, by \cite[Proposition 10.17]{KunzKahlerDifferentials}, we know $ \sD_{S/R} =
\Fitt^{0}(\Omega_{S/R})$ and thus $\langle s \rangle = \langle a_{1}
\cdots a_{m} \rangle$.  Hence, for any height one (\ie non-zero)
prime ideal $\bq$ of $S$ corresponding to a prime divisor $E$ on $Y$, we have
\smallskip
\begin{equation}
\label{eqn:DifferentDifferents}
\begin{array}{c@{\,\,=\,\,}c@{\,\,=\,\,}c}
\length_{S_{\bq}}(\Omega_{S/R,\bq}) & \sum_{i=1}^{m}
\length_{S_{\bq}}(S_{\bq}/a_{i}S_{\bq}) &\sum_{i=1}^{m} \ord_{S_{\bq}}(a_{i}) \\[8pt]
& \ord_{S_{\bq}}\left(\prod_{i=1}^{m}a_{i}\right)&\ord_{S_{\bq}}(s)
\end{array}
\end{equation}
\smallskip
where $\ord_{S_{\bq}}(\blank)$ is the valuation of the DVR $S_{\bq}$.  As
the left and right sides of (\ref{eqn:DifferentDifferents}) determine the order of $E$ in
$\Ram_{\pi}$ and $\Div(S)$, respectively, the statement follows.
\end{proof}

\begin{remark}
  Using the notation from the proof above, alternate demonstrations
  that $$\ord_{S_{\bq}}(\sD_{S/R,\bq}) = \length_{S_{\bq}}(\Omega_{S/R,\bq})$$
  can be found in \cite{deSmitTheDifferentAndDifferentials} or \cite{SchejaStorchUberSpurfunktionen}, and appeared as early as \cite{MoriyaTheorieDerDerivationen}.  This fact also follows from \cite[Proposition III.7.14]{SerreLocalFields}
  when $S$ is monogenic over $R$.
\end{remark}

\begin{remark} Let $R \subseteq S$ be as above.
  In the language of \cite{KunzKahlerDifferentials}, the different
  ideal $\sD_{S/R}$
  defined using the trace map is known as the Dedekind different, while
 the ideal
  $\Fitt^{0}(\Omega_{S/R})$ is referred to as the K{\"a}hler
  different.  The crucial statement needed above is that
  these agree in the traditional settings of algebraic number theory.
\end{remark}

\section{Extending $p^{-e}$-linear maps over arbitrary inclusions}

Given a module-finite inclusion of domains $R \subseteq S$ which is not
generically separable, one may still ask whether an $R$-linear map
$\phi : F^e_* R \to R$ extends to an $S$-linear map $\bar{\phi} :
F^e_* S \to S$.  However, consider the following.

\begin{example}
\label{ex:pureinsep}
Suppose that $K \subseteq L$ is an extension of fields in characteristic $p > 0$ such that $L$ contains an element which is purely inseparable over $L$ (for example, if $K \subseteq L$ itself is purely inseparable).  We will show that $\phi \in \Hom_{K}(K^{1/p^{e}},K)$ extends to $\bar{\phi} \in \Hom_{L}(L^{1/p^{e}},L)$ if and only if $\phi$ is the zero map.

Choose $0 \neq x \in K$ such that $x^{1/p} \in L \setminus K$.  Thus we also have $y = x^{p^{e-1}} \in K$ with $y^{1/p^e} = x^{1/p} \in L \setminus K$.
Suppose that $\phi$ is not the zero map and pick any $z \in K$ such that $\phi(z^{1/p^e}) = u \neq 0$.  Assume, by way of contradiction, that $\phi$ extends to an $L$-linear map $\bar{\phi} \: L^{1/p^e} \to L$.  Then, since $y^{1/p^{e}}z^{1/p^{e}} \in K^{1/p^{e}}$, we have
\[
y^{1/p^e} u = y^{1/p^e} \phi(z^{1/p^e}) = y^{1/p^{e}} \bar{\phi}(z^{1/p^{e}}) =  \bar{\phi}(y^{1/p^e} z^{1/p^e}) = \phi(y^{1/p^{e}}z^{1/p^{e}}) \in K.
\]
But this implies that $y^{1/p^e} \in K$, which is a contradiction.
\end{example}

In fact, the line of argument given in the example above implies the following.

\begin{proposition}
\label{prop:nonsepnoextend}
Let $R \subseteq S$ be a module-finite inclusion of normal $F$-finite domains and suppose that the associated inclusion of fraction fields $K \subseteq L$ is not separable.  Then the only map $\phi \in \Hom_R(R^{1/p^e}, R)$ that extends to an element of $\Hom_S(S^{1/p^e}, S)$ is the zero map.
\end{proposition}
\begin{proof}
We may immediately reduce to the case where $R = K$ and $S = L$ are fields.  Suppose that $\phi \in \Hom_K(K^{1/p^e}, K)$ extends to $\bar{\phi} \in \Hom_L(L^{1/p^e}, L)$.  We will reduce to the case where $L$ contains an element of $K \setminus K^p$, at which point we can apply Example \ref{ex:pureinsep}.  To that end, let $K'$ denote the separable closure of $K$ in $L$.  We will show that $\bar{\phi}(K'^{1/p^e}) \subseteq K'$ (which will complete the proof by allowing us to replace $K$ by $K'$).  Choose $x \in K'$.  Since we know that $K'^{1/p^e} = K^{1/p^e} K'$, by Lemma \ref{lem:genericextension}, we can write $x^{1/p^e} = f_1 g_1^{1/p^e} +  \dots + f_n g_n^{1/p^e}$ for $f_i \in K'$ and $g_i \in K$.  Then since $\bar{\phi}$ is $L$-linear, it is certainly $K'$-linear so that
\[
\bar{\phi}(x^{1/p^e}) = \bar{\phi}(f_1 g_1^{1/p^e} +  \dots + f_n g_n^{1/p^e}) = f_1 \bar{\phi}(g_1^{1/p^e}) + \dots + f_n \bar{\phi}(g_n^{1/p^e}) \in K'
\]
as desired.
\end{proof}

Therefore, in light of Proposition~\ref{prop:nonsepnoextend}, the
question of extending $p^{-e}$-linear maps is uninteresting for ring
inclusions which are not generically separable.  However,
Proposition~\ref{sec:commtrace} suggests a fruitful way to alter our
investigations in this setting.

\begin{definition}
Suppose that $\pi : Y \to X$ finite surjective morphism of irreducible
schemes and that $\Tt : \pi_* K(Y) \to K(X)$ is an $\O_X$-linear map.
For a line bundle $\sL$ on $X$, we say that an $\O_X$-linear map $\phi
: \sL^{1/p^e} \to \O_X$ \emph{has a transpose along $\Tt$} if there exists a map $\phi_{\Tt} : (\pi^* \sL)^{1/p^e} \to \O_Y$ such that the following diagram commutes.
\[
\xymatrix@C=60pt{
 \sL^{1/p^e} \tensor_{\O_X^{1/p^e}} K(X)^{1/p^e}       \ar^-{\phi \tensor_{\O_X} {K(X)} }[r]    &     K(X)  \\
 \pi_* (\pi^* \sL)^{1/p^e}    \ar_-{\phi_{\Tt}}[r]    \ar[u]^{\alpha}             &     \pi_* \O_Y   \ar_{\Tt}[u]
}
\]
Here the left vertical map, labeled $\alpha$ is constructed as
follows:  restrict the domain of $\Tt^{1/p^e} : \pi_* K(Y)^{1/p^e} \to
K(X)^{1/p^e}$ to $\pi_* \O_Y^{1/p^e}$ and then tensor with
$\sL^{1/p^e}$, noting that $\pi_* \O_Y^{1/p^e} \tensor_{\O_X^{1/p^e}}
\sL^{1/p^e} \cong \pi_* (\pi^* \sL)^{1/p^e}$.  The map $\phi_{\Tt}$ is
called the \emph{$\Tt$-transpose of $\phi$}.
\end{definition}

\begin{proposition}
\label{propGenericLifting}
Suppose we have a finite extension of fields $K \subseteq L$ and a
non-zero \mbox{$K$-linear} map $\Tt : L \to K$.  Then any $K$-linear map
$\phi : K^{1/p^e} \to K$ has a unique $\Tt$-transpose \mbox{$\phi_{\Tt} : L^{1/p^e} \to L$}.
\end{proposition}
\begin{proof}
Consider any non-zero map $\psi : L^{1/p^e} \to L$ and the (likely non-commutative) diagram
\[
\xymatrix{
K^{1/p^e}       \ar^{\phi}[r]    &     K   \\
L^{1/p^e}    \ar_{\psi}[r]     \ar^{(\Tt)^{1/p^e}}[u]            &     L   \ar_{\Tt}[u]
}
\]
We know $\Hom_K(L^{1/p^e}, K)$ is abstractly isomorphic to $L^{1/p^e}$ as an $L^{1/p^e}$-module.  Therefore, there exists an element $z \in L$ such that
\[
\Tt \circ \psi(z^{1/p^e} \cdot \blank) = \phi \circ \Tt^{1/p^e}(\blank).
\]
It follows that $\phi_{\Tt}(\blank) := \psi(z^{1/p^e} \cdot \blank)$ is the unique map making the diagram commute.
\end{proof}

\begin{corollary}
Suppose that $K \subseteq L$ is a finite separable extension of fields and that \mbox{$\Tr : L \to K$} is the trace map.  Fix a $K$-linear map $\phi : K^{1/p^e} \to K$.  Then the (unique) extension of $\phi$ to an $L$-linear map, $\overline{\phi} : L^{1/p^e} \to L$ is equal to the (unique) $\Tr$-transpose $\phi_{\Tr} : L^{1/p^e} \to L$.
\end{corollary}
\begin{proof}
This follows from Lemma \ref{lem:genericextension}, Proposition \ref{prop:fieldtracecommute} and Proposition \ref{propGenericLifting}.
\end{proof}

The existence of a $\Tt$-transpose can also be checked in codimension one, \cf Lemma \ref{lem:reductcodimoneaffine}.

\begin{proposition}
\label{prop:LiftingIffCodim1}
Suppose we have a finite surjective map $\pi : Y \to X$ of irreducible
schemes and a non-zero $\O_X$-linear map $\Tt : \pi_* K(Y) \to K(X)$.
Given an $\O_X$-linear map $\phi : \sL^{1/p^e} \to \O_X$, there exists
a transpose of $\phi$ along $\Tt$ if and only if for every codimension
one point $\eta \in X$, the induced map $\phi_{\eta} : \sL_{\eta}^{1/p^e} \to \O_{X, \eta}$ has a transpose along $\Tt$.
\end{proposition}
\begin{proof}
We already know that the $\Tt$-transpose of $\phi$ exists generically
as $\phi_{\Tt}$, so one only has to check whether
$\phi_{\Tt}( \pi_{*}(\pi^* \sL)^{1/p^e}) \subseteq \pi_{*}\O_Y$.  This
in turn can
be checked at the stalks of codimension one points because $Y$ is normal.
\end{proof}

\subsection{Transpose criterion}
\label{sec:LiftingCriterionTheorem}

\begin{theorem}
\label{thm:mainliftingcriterion}
 Suppose $\pi \: Y \to X$ is a finite surjective morphism
of irreducible $F$-finite normal schemes
and $0 \neq \phi \in \Hom_{\O_{X}}(\sL^{1/p^e}, \O_{X})$ where $\sL$ is an
invertible $\O_{X}$-module.  Let $\Delta_{\phi}$ be the
$\Q$-divisor on $X$ associated to $\phi$ via
\eqref{eq:divmaplbcorr}.  Fix a non-zero $\O_{X}$-linear map $\Tt : \pi_* K(Y) \to K(X)$ and set $\RamiT$ to be the divisor corresponding to $\Tt|_{\pi_* \O_Y}$ via \eqref{eq:mapstodivisorsgeneral}.  Then:
\begin{itemize}
\item[(a)]  $\phi$ has a $\Tt$-transpose $\phi_{\Tt} \in
  \Hom_{\O_Y}(\pi^* \sL^{1/p^e}, \O_Y)$
if and only if $\pi^{*} \Delta_{\phi} \geq \RamiT$.
\item[(b)]  In the case of (a), the $\bQ$-divisor
  $\Delta_{\phi_{\Tt}}$ associated to $\phi_{\Tt}$ via
  \eqref{eq:divmaplbcorr} is equal to $\pi^{*} \Delta_{\phi} -
  \RamiT$. In particular, it follows that $K_{Y} +
  \Delta_{\phi_{\Tt}} \sim_{\Q} \pi^{*}(K_{X} + \Delta_{\phi})$.
\end{itemize}
Furthermore, if $\pi$ is generically separable one may take $\Tt = \Tr_{Y/X}$ so that $\RamiT = \Ram_{\pi}$.
\end{theorem}

\begin{proof}
  From Proposition \ref{prop:LiftingIffCodim1}, we may assume that $X =
  \Spec R$ for $R$ a DVR, $Y = \Spec S$ for $S$ a Dedekind domain,
  $R \subseteq S$ a module-finite extension with $K =
  K(R) \subseteq K(S) = L$ the induced finite extension of fraction fields, and $\sL \cong \O_{X} = R$.
  If we identify $\phi \in \Hom_{R}(R^{1/p^e},R)$ with its generic
  extension $\phi \: K^{1/p^e} \to K$ and denote by $\phi_{\Tt} \: L^{1/p^e}
  \to L$ the generic transpose along $\Tt$, then we need to show
  $\phi_{\Tt}(S^{1/p^e}) \subseteq S$ if and only if $\pi^{*} \Delta_{\phi} \geq
  \RamiT$.

  Choose an $R^{1/p^e}$-module generator $\Psi_{R}$ for $\Hom_{R}(R^{1/p^e},R)
  \simeq R^{1/p^e}$
  and an $S^{1/p^e}$-module generator $\Psi_{S}$ for $\Hom_{S}(S^{1/p^e},S)
  \simeq S^{1/p^e}$.  We can write $\phi(\blank) = \Psi_{R}(x^{1/p^e} \cdot \blank)$ for
  $x \in R$ and $\phi_{\Tt}(\blank) = \Psi_{S}(y^{1/p^e} \cdot \blank)$ for $y \in
  L$.  It follows that $\phi_{\Tt}(S) \subseteq S$ if and only if $y \in
  S$, \ie $\Div_{Y}(y) \geq 0$.  Also choose a generator $\Phi$ for $\Hom_{R}(S,R)$ as an
  $S$-module, and note that we can identify $\Tt$ with $\Phi(s \cdot \blank)$ for some $s
  \in L$.  We know $\Div_{Y}(s) = \RamiT$.  Also observe that $\Phi^{1/p^e}$ generates
  $\Hom_{R^{1/p^e}}(S^{1/p^e},R^{1/p^e}) \simeq S^{1/p^e}$ as an
  $S^{1/p^e}$-module, and $\Tt^{1/p^e} =
  \Phi^{1/p^e}(s^{1/p^e}\cdot \blank)$.

Working at the level of fraction fields, we have
  \[
    \Tt \circ \, \, \phi_{\Tt} = \phi \circ \Tt^{1/p^e}
  \in \Hom_{K}(L^{1/p^e},K) \simeq L^{1/p^e}.
  \]
Thus
$\Phi( \Psi_S(s \cdot y^{1/p^e} \cdot \blank) = \Phi(s
\cdot\Psi_S(y^{1/p^e} \cdot \blank)$ agrees with $\Psi_R(x^{1/p^e}
\Phi^{1/p^e}(s^{1/p^e} \cdot \blank)) = \mbox{$\Psi_R( \Phi^{1/p^e}(x^{1/p^e}s^{1/p^e} \cdot
\blank))$}$.  Now both $\Phi \circ \Psi_{S}$ and $\Psi_R \circ \Phi^{1/p^e}$ are \mbox{$S^{1/p^e}$-module} generators for $\Hom_{R}(S^{1/p^e},R)$ by \cite[Appendix F]{KunzKahlerDifferentials} \cite[Lemma 3.9]{SchwedeFAdjunction}.
Therefore $\Div_{Y}(s^{p^e} y)  = \Div_{Y}(xs)$ so that
  \[
    \Delta_{\phi_{\Tt}} = \frac{1}{p^e-1} \Div_{Y}(y) =  \frac{1}{p^e-1} \Div_{Y}(x) -
    \Div_{Y}(s)
    = \pi^{*} \Delta_{\phi} - \RamiT \, \, .
  \]
  Thus, $\phi_{\Tt}(S^{1/p^e}) \subseteq S$ if and only if
  $\pi^{*} \Delta_{\phi} - \RamiT \geq 0$.
 The remaining statement follows from $\RamiT \sim K_Y - \pi^* K_X$.
  \end{proof}


\section{The behavior of the test ideal under finite morphisms}
\label{sec:testideals}

\subsection{Test ideals, $F$-regularity, and $F$-purity}
\label{sec:testidealsdefinitions}

Given a ring $R$ of characteristic $p > 0$, the test ideal $\tau(R)
\subseteq R$ measures the singularities of $R$.  Test ideals have been
a fundamental object of study in commutative algebra for approximately
25 years, and as mentioned in the introduction, are closely related to
multiplier ideals.  In this section we apply the techniques of Section
\ref{sec:LiftingCriterionTheorem} to describe the behavior of the test
ideal under finite integral ring extensions.

The definition of the test ideal has in fact been revisited many times
over, and we focus herein on a slight variant of that which was originally defined in
\cite{HochsterHunekeTC1}. For clarity, what we are interested in is
now often referred to as the \emph{big test ideal}\footnote{The big
  test ideal is sometimes known as the non-finitistic test ideal, in
  contrast to the classical or finitistic test ideal.} and is denoted
$\tau_{b}(R)$.  It is a geometric object whose formation is known to
commute with localization and completion (in contrast to the classical or
finitistic test ideal).  In the simple setting of an $F$-finite normal domain $R$, the big test ideal $\tau_b(R)$ is the smallest non-zero ideal $J \subseteq R$ such that $\phi(J^{1/p^e}) \subseteq J$ for all maps $\phi \in \Hom_R(R^{1/p^e}, R)$ and all $e \geq 0$.
Before giving a more general definition we need some additional language.

\begin{remark}The big test ideal and the finitistic test ideal
  have been shown to agree in many cases, for example the
  $\bQ$-Gorenstein case; see
  \cite{LyubeznikSmithCommutationOfTestIdealWithLocalization} and
  \cite{AberbachMacCrimmonSomeResultsOnTestElements}.  See
  Remark~\ref{rem:TwoIdealsAgree} for the precise statement in full
  generality.
\end{remark}

\begin{definition}
\label{defn:triple}
 A \emph{triple} $(X, \Delta, \ba^t)$ is the combined information of a normal integral $F$-finite scheme $X$ satisfying  (\ref{eqn:FUpperShriekOmegaIsOmega}), a $\bQ$-divisor $\Delta$ (usually assumed to be effective), a non-zero ideal sheaf $\ba$, and a nonnegative real number $t \geq 0$.  By a \emph{pair} $(X, \Delta)$ we simply mean a triple where $\ba = R$ and $t = 1$.
\end{definition}

Our presentation of the (big) test ideal is somewhat non-standard, but has the advantage that one need not develop a theory of tight closure in order to state the definition.

\begin{definition} \cite{HochsterHunekeTC1, HochsterFoundations, LyubeznikSmithCommutationOfTestIdealWithLocalization, HaraYoshidaGeneralizationOfTightClosure, TakagiInterpretationOfMultiplierIdeals, SchwedeCentersOfFPurity}
\label{DefnBigTestIdeal}
For an affine triple $(X = \Spec R, \Delta, \ba^{t})$ with $\Delta \geq 0$,
 the\emph{ big test ideal} $\tau_b(X; \Delta, \ba^t)$ is the unique smallest non-zero ideal $J$ of $R$ satisfying
\[
 \phi \left(\Frobp{\ba^{\lceil t(p^e - 1) \rceil} J}{e} \right) \subseteq J
\]
for all $e \geq 0$ and all $\phi \in \Hom_{R}(R^{1/p^{e}}, R)$ such
that $\Delta_{\phi} \geq \Delta$ via \eqref{eq:divmapcorr}, \ie   such that
\[
\phi \in \Image\Bigg(\Hom_R\left( \Frobp{R(\lceil (p^e - 1)\Delta \rceil)}{e}, R\right) \rightarrow \Hom_R(\Frob{R}{e}, R) \Bigg).
\]
\end{definition}

This ideal always exists in the above context, as will be reviewed
shortly in
Lemma \ref{Lemma:ExistOfTestElts}.  First, however, note that in many common
situations the above definition can be simplified substantially --
as seen below.

\begin{lemma} \label{lem:SimpleTestIdeal}\cite[Proposition 4.8]{SchwedeFAdjunction}
Suppose that $(X = \Spec R, \Delta, \ba^{t})$ is an 
affine triple with $\Delta \geq 0$, and further assume that $\Delta = \Delta_{\phi}$ for
some $\phi \in \Hom_R(\Frob{R}{e}, R)$ via (\ref{eq:divmapcorr}).  Set
$\phi^n$ to be the $n$-th Frobenius iterate of $\phi$ as in (\ref{eqn:ComposeMapWithSelf}).  Then  $\tau_b(X; \Delta, \ba^t)$ is the unique smallest ideal $J \neq 0$ of $R$ such that
\[
 \phi^n \left(\Frobp{\ba^{\lceil t(p^{ne} - 1) \rceil} J}{ne} \right) \subseteq J
\]
for all $n \geq 0$.
\end{lemma}

The following lemma is a technical tool which guarantees that
Definition \ref{DefnBigTestIdeal} is well-posed.  In fact, in
Proposition~\ref{prop:ConstructionOfTestIdeal} we will see that it
provides a recipe for constructing big test ideals.   We have omitted a proof as we will
not make use of the techniques in what follows;  the argument is essentially the same as that which shows test elements exist in $F$-finite reduced rings.

\begin{lemma} \cite{HochsterHunekeTC1}, \cite[Remark 3.2]{HochsterHunekeFRegularityTestElementsBaseChange}, \cite[Proposition 6.14]{SchwedeFAdjunction}  \label{Lemma:ExistOfTestElts}
 Let $X = \Spec R$ be an $F$-finite normal integral affine scheme,
 $\ba \neq 0$ an ideal sheaf, and $t \geq 0$ a real number.  Suppose that $\psi \in \Hom_R(\Frob{R}{e}, R)$ is a non-zero element, and denote by $\psi^{n}$ the $n$-th Frobenius iterate of $\psi$ as in (\ref{eqn:ComposeMapWithSelf}).   Then there exists an element $c \in R\setminus \{0\}$ such that for every $d \in R\setminus \{0\}$ there is an integer $n > 0$ with
\begin{equation} \label{EqnHomogTestElement}
 c \in \psi^{n} \left(\Frobp{d \ba^{\lceil t(p^{ne} - 1) \rceil} }{ne} \right).
\end{equation}
\end{lemma}

\begin{remark}
Note that if $c \in R$ satisfies the conclusion of Lemma \ref{Lemma:ExistOfTestElts}, then for any $c' \in R\setminus \{0\}$, the product $c' c$ also satisfies the conclusion of Lemma \ref{Lemma:ExistOfTestElts}.
\end{remark}

\begin{remark}
\label{RemarkExistenceHelp}
If $c \in \ba \cap R\setminus \{0\}$ is such that $R_c$ is regular and $\psi/1$ generates $\Hom_{R_c}(\Frob{R_c}{e}, R_c)$ as an $\Frob{R_c}{e}$-module, then some power of $c$ can be shown to satisfy the conclusion of Lemma \ref{Lemma:ExistOfTestElts} for $R$.
\end{remark}

We will also use the following notation throughout.
\begin{definition}
\label{DefnIDeltaJDelta}
Given an affine pair $(X = \Spec R, \Delta)$ with $\Delta \geq 0$,
for convenience we set
\[
\begin{array}{rcl}
I_{\Delta, \, e} &= & \Big\{ \left. \, \, \phi \in \Hom_R(\Frob{R}{e}, R) \, \, \right| \, \,  \Delta_{\phi} \geq \Delta \, \, \Big\}\\
& = &
\Image\Bigg(\Hom_R\left(\Frobp{R(\lceil (p^e - 1)\Delta \rceil)}{e}, R\right) \rightarrow \Hom_R(\Frob{R}{e}, R) \Bigg)
\end{array}
\]
and
\[
\begin{array}{rcl}
J_{\Delta,\, e} &=&  \Big\{ \left. \, \, \phi \in \Hom_R(\Frob{R}{e}, R) \, \, \right| \, \,  \Delta_{\phi} \geq \frac{1}{p^{e}-1} \lfloor (p^{e}-1) \Delta \rfloor \, \, \Big\}\\
& = & \Image\Bigg(\Hom_R\left(\Frobp{R(\lfloor (p^e - 1)\Delta \rfloor)}{e}, R\right) \rightarrow \Hom_R(\Frob{R}{e}, R) \Bigg).
\end{array}
\]
\end{definition}

\begin{proposition}
\label{prop:ConstructionOfTestIdeal}
Suppose $(X = \Spec R, \Delta, \ba^{t})$ is an
affine triple with $\Delta \geq 0$.  Choose any non-zero map $\psi \in \Hom_R(\Frob{R}{e}, R)$ such that $\Delta_{\psi} \geq \Delta$ and following Lemma \ref{Lemma:ExistOfTestElts} pick an element $c \in R\setminus \{0\}$ satisfying (\ref{EqnHomogTestElement}).  Then the big test ideal $\tau_b(X, \Delta, \ba^t)$ is given by
\[
 \tau_b(X; \Delta, \ba^t) = \sum_{e \geq 0} \left( \sum_{\phi \in I_{\Delta, e}} \phi \left(\Frobp{\ba^{\lceil t(p^e - 1) \rceil} c}{e} \right) \right).
\]
\end{proposition}
\begin{proof}
Clearly the sum above contains the element $c \neq 0$ (set $d = c$ in Lemma \ref{Lemma:ExistOfTestElts}), and so is a non-zero ideal.
 On the other hand, note that any non-zero ideal $J$ satisfying the conditions of Definition \ref{DefnBigTestIdeal} will certainly contain $c$.  Thus, we see that the sum is the unique smallest ideal $J$ which satisfies the conditions of Definition \ref{DefnBigTestIdeal}.
\end{proof}

\begin{remark}
 Suppose that $c$ satisfies the condition of Lemma
 \ref{Lemma:ExistOfTestElts}.  Then for any multiplicative set $W$ of
 $R$ with $Y = \Spec W^{-1} R$, $c/1$ also satisfies the condition of
 Lemma \ref{Lemma:ExistOfTestElts} for the triple $(Y, \Delta|_Y,
 (W^{-1} \ba)^t)$.  It immediately follows that the formation of
 $\tau_b(X, \Delta, \ba^t)$ commutes with localization,
 \cite{HaraTakagiOnAGeneralizationOfTestIdeals}.  In particular, we
 may define $\tau_b(X, \Delta, \ba^t)$ when $X$ is a non-affine scheme
 by gluing over affine charts.
\end{remark}

\begin{remark}
\label{rem:TwoIdealsAgree}  \cite[Theorem 2.8]{TakagiInterpretationOfMultiplierIdeals}, \cite[Proposition 3.7]{BlickleSchwedeTakagiZhang}
Because finitistic test ideals $\tau(X, \Delta, \ba^t)$ appear so frequently in the literature (\eg see \cite{HaraYoshidaGeneralizationOfTightClosure}), we recall the following result.
Consider an 
affine triple $(X = \Spec R, \Delta, \ba^{t})$ with $\Delta \geq 0$.
If $K_X + \Delta$ is $\Q$-Cartier, then the ideals $\tau(X;\Delta, \ba^t)$ and $\tau_b(X;\Delta, \ba^t)$ coincide.
\end{remark}

We will also need a definition for test ideals $\tau_b(X; \Delta, \ba^t)$ under the hypothesis that $\Delta$ is not necessarily effective.  To that end, let us first recall a variant of Skoda's theorem for test ideals  (\cf \cite[Section 9.6]{LazarsfeldPositivity2}).  We include a proof both for the sake of completeness and also because our notation differs slightly from that in the literature.

\begin{lemma} \cite[Theorem 4.1]{HaraTakagiOnAGeneralizationOfTestIdeals}
\label{Lemma:SkodaTypeTheorem}
Suppose that $(X = \Spec R, \Delta, \ba^t)$ is an affine triple with $\Delta \geq 0$.
Assume that $f \in R$ is any non-zero element.  Then
\[
 \tau_b(X; \Delta + \Div(f), \ba^t) = f \cdot \tau_b(X; \Delta, \ba^{t}) .
\]
\end{lemma}
\begin{proof}
 One may choose $c' \in R\setminus \{0\}$ satisfying the conditions of
 Lemma \ref{Lemma:ExistOfTestElts} for both of the triples $(X, \Delta + \Div(f), \ba^t)$ and $(X, \Delta, \ba^t)$.
Then observe that
\[
\begin{array}{rcl}
 f \cdot \tau_b(X; \Delta, \ba^{t})
&=& f \cdot \sum_{e \geq 0} \left( \sum_{\phi \in I_{\Delta, e}} \phi \left(\Frobp{c' \ba^{\lceil t(p^e - 1) \rceil} }{e} \right) \right) \\
&=& \sum_{e \geq 0} \left( \sum_{\phi \in I_{\Delta, e}} \phi \left(\Frobp{c' f^{p^e} \ba^{\lceil t(p^e - 1) \rceil}}{e} \right) \right) \\
&=& \sum_{e \geq 0} \left( \sum_{\phi \in I_{\Delta + \Div(f), e}} \phi \left(\Frobp{ (c' f) \ba^{\lceil t(p^e - 1) \rceil} }{e} \right)\right) \smallskip \\
&=& \tau_b(X; \Delta + \Div(f), \ba^t)
\end{array}
\]
as desired.
\end{proof}

Inspired by the above result, we define the big test ideal for non-effective divisors as follows.

\begin{definition}
 Suppose that $(X = \Spec R, \Delta, \ba^t)$ is an affine triple where $\Delta$ is a possibly non-effective divisor.  Choose an element $f \in R$ such that $\Div(f) + \Delta$ is effective.  We then define $\tau_b(X; \Delta, \ba^t)$ to be the fractional ideal
\[
 \tau_b(X; \Delta, \ba^t) = {1 \over f} \tau_b(X; \Delta + \Div(f), \ba^{t}).
\]
It is not difficult to verify that this definition is independent of the choice of $f \in R$.  When $X$ is not affine, $\tau_b(X; \Delta, \ba^{t})$ is given by gluing over a cover of $X$ by affine charts.
\end{definition}

We conclude this section with the definitions of strongly $F$-regular and sharply $F$-pure triples. These classes of singularities are analogous, respectively, to log terminal and log canonical singularities for complex algebraic varieties; see \cite{HaraWatanabeFRegFPure}.

\begin{definition}
 If $R$ is a local ring, then a triple $(X = \Spec R, \Delta, \ba^t)$ with $\Delta \geq 0$ is called \emph{strongly $F$-regular} if, for every $d \in \Spec R$, there exists an $e > 0$ and  $\phi \in I_{\Delta, e}$ such that $$1 \in \phi\left(\Frobp{d \ba^{\lceil t(p^e - 1) \rceil} }{e}\right) \, \, .$$
More generally, an arbitrary triple $(X, \Delta, \ba^t)$ is called \emph{strongly $F$-regular} if, for every closed point $x \in X$, the triple associated to the stalk $\O_{X, x}$ is strongly $F$-regular.
\end{definition}

\begin{definition}
 If $R$ is a local ring, a triple $(X = \Spec R, \Delta, \ba^t)$ with $\Delta \geq 0$ is called \emph{sharply $F$-pure} if there exists an $e > 0$ and a $\phi \in I_{\Delta, e}$ such that $$1 \in \phi\left(\Frobp{ \ba^{\lceil (p^e - 1) \rceil} }{e}\right) \, \, .$$
 More generally,  an arbitrary triple $(X, \Delta, \ba^t)$ is called \emph{sharply $F$-pure} if, for every closed point $x \in X$, the triple associated to the stalk $\O_{X, x}$ is sharply $F$-pure.
\end{definition}

\begin{remark}
$F$-pure rings are also sometimes called \emph{(locally) $F$-split}.
\end{remark}

\begin{remark}
 In case $X$ is not the spectrum of a local ring (even in the case
 that $X$ is affine) the definitions of strongly $F$-regular and
 sharply $F$-pure triples given here differ from those in the literature; see
 \cite{TakagiInversion} and \cite{SchwedeSharpTestElements}.  However,
 the above definitions possess more desirable behavior with respect to localization; see \cite{SchwedeBetterFPureFRegular} for additional discussion.
\end{remark}

The following proposition collects together some well known facts about sharp $F$-purity and strong $F$-regularity.

\begin{proposition} \cite[Lemma 1.6]{FedderFPureRat}, \cite[Lemma 2.8]{SchwedeCentersOfFPurity}, \cite[Proposition 2.2(5)]{TakagiWatanabeFPureThresh}
Suppose that $R$ is a local ring and $(X = \Spec R, \Delta, \ba^t)$ is a triple with $\Delta \geq 0$.
\begin{itemize}
\item[(i)]  Then $(X, \Delta, \ba^t)$ is sharply $F$-pure if and only if there exists an integer $n_0 > 0$ such that for all integers $m > 0$ there is some $\phi \in I_{\Delta, m n_0}$ with
    \[
    \phi \left(\Frobp{\ba^{\lceil t(p^{m n_0} - 1) \rceil} }{m n_0} \right) = R.
    \]
\item[(ii)]  Suppose that $R$ is regular, $\Supp \Delta$ is a normal crossings divisor, and $\lceil \Delta \rceil$ is a reduced integral divisor.  Then $(X, \Delta)$ is sharply $F$-pure.
\item[(iii)]  Suppose that $R$ is regular, $\Supp \Delta$ is a normal crossings divisor, and $\lfloor \Delta \rfloor = 0$.  Then $(X, \Delta)$ is strongly $F$-regular.
\item[(iv)]  The triple $(X, \Delta, \ba^t)$ is strongly $F$-regular if and only if $\tau_b(X; \Delta, \ba^t) = R$.
\end{itemize}
\end{proposition}
\begin{proof}
The proof of (i) mirrors \cite[Lemma
2.8]{SchwedeCentersOfFPurity}; see also \cite[Proposition 3.3]{SchwedeSharpTestElements}.
For (ii), we note that it is sufficient to prove the statement when $\Delta$ itself is a reduced simple normal crossings divisor by the argument of \cite[Proposition 2.2(3)]{TakagiWatanabeFPureThresh} or \cite[Lemma 5.2]{SchwedeSharpTestElements}.  Normal crossing divisors are $F$-pure and so, by Fedder's criterion \cite[Lemma 1.6]{FedderFPureRat}, there exists a surjective map $\phi \: \Frob{R}{e} \rightarrow R$ which restricts to a map $\phi' \: \Frob{(R(-\Delta))}{e} \to R(-\Delta)$.  Twisting both sides by $R(\Delta)$ and applying the projection formula gives us a map
\[
 \Frob{\left(R((p^e - 1) \Delta)\right)}{e} \rightarrow R
\]
whose restriction to $R$ is $\phi$.  Thus (ii) is complete.
Furthermore, (iii) follows immediately from (ii) and the argument of \cite[Proposition 2.2(5)]{TakagiWatanabeFPureThresh}.

The forward direction of (iv) is clear:  if $(X, \Delta, \ba^t)$ is strongly $F$-regular, then by definition the element $1$ is in any non-zero ideal $J$ satisfying
\[
\phi \left(\Frobp{\ba^{\lceil t(p^e - 1) \rceil} J}{e} \right) \subseteq J
\]
for all $e \geq 0$ and all $\phi \in I_{\Delta_X, e}$.  For the converse, notice that for any $d \in R \setminus \{0 \}$, the ideal
\[
\sum_{e \geq 0 }  \sum_{\phi \in I_{\Delta_X, e}} \phi \left(\Frobp{d\ba^{\lceil t(p^e - 1) \rceil}}{e} \right)
\]
contains $\tau_b(X, \Delta, \ba^t)$ by Lemma \ref{Lemma:ExistOfTestElts}.  Thus if $\tau_b(X; \Delta, \ba^t) = R$, since $R$ is local, we must have
$ \phi \left( \Frobp{d\ba^{\lceil t(p^e - 1) \rceil}}{e} \right)  = R$ for some $e \geq 0$ and $\phi \in I_{\Delta_X, e}$, which completes the proof.
\end{proof}

\subsection{Test ideals and finite maps}

Throughout this section, we will be in the following situation.

\begin{setting}
\label{set.TestIdealsSetting}
Let $\pi \: Y \rightarrow X$ be a surjective finite morphism of $F$-finite normal
schemes satisfying (\ref{eqn:FUpperShriekOmegaIsOmega}) and $(X, \Delta_X, \ba^t)$ a triple.  Fix a non-zero
$\O_{X}$-linear homomorphism $\Tt : \pi_* K(Y) \to K(X)$ and set
$\RamiT$ to be the integral divisor on $Y$ corresponding to $\Tt|_{\pi_* \O_Y}$ as
in \eqref{eq:mapstodivisorsgeneral}.
We set
\[
\Delta_{Y/\Tt} = \pi^*\Delta_X - \RamiT
\]
and will consider the triple $(Y, \Delta_{Y/\Tt},  (\ba \cdot \O_{Y})^{t})$.

\end{setting}

Our main goal is to show that
\[ \Tt\left(\tau_b(Y; \Delta_{Y/\Tt}, (\ba \cdot \O_{Y})^{t})\right) = \tau_b(X; \Delta_X, \ba^{t}).\]  This will be done in two steps, and we proceed first with the easier inclusion.

\begin{proposition}
\label{Prop:EasyTestIdealContainment}
Suppose we are in the setting of \ref{set.TestIdealsSetting}.   Then
\[
\Tt\left(\tau_b(Y; \Delta_{Y/\Tt}, (\ba \cdot \O_Y)^t) \right) \supseteq \tau_b(X, \Delta_X, \ba^t).
\]
\end{proposition}
\begin{proof}
The statement is local on $X$ and so we assume that $X = \Spec R$ where $R$ is a local ring and $Y = \Spec S$ for a semi-local ring $S$.  By restricting the target, we may view $\Tt$ as a map $\Tt : S \to R(E)$ for some effective Cartier divisor $E = \Div_X(g)$.   Composing with the isomorphism ${1 \over g} R \cong R$, we have a map $\Tt' : S \to R$.  Notice that $\RamiT = \mathcal{R}_{\Tt'} - \pi^* E$.  Thus
\[
\begin{array}{lll}
\Tt\left(\tau_b(Y; \Delta_{Y/\Tt}, (\ba \cdot \O_Y)^t) \right) & = & {1 \over g} \Tt'\left(  \tau_b(Y; \Delta_{Y/\Tt}, (\ba \cdot \O_Y)^t) \right) \\
 & = & \Tt'\left(  \tau_b(Y; \Delta_{Y/\Tt} - \pi^*E, (\ba \cdot \O_Y)^t) \right) \\
 & = & \Tt'\left(  \tau_b(Y; \Delta_{Y/\Tt'}, (\ba \cdot \O_Y)^t) \right)
 \end{array}
\]
where the penultimate equality uses Lemma \ref{Lemma:SkodaTypeTheorem}.  Thus we may assume that $\Tt = \Tt'$ and view $\Tt \in \Hom_R(S, R)$.
Likewise, it is sufficient to show that $f \cdot {\Tt}(\tau_b(Y;
\Delta_{Y/\Tt}, (\ba \cdot \O_Y)^t) ) \supseteq f \cdot \tau_b(X, \Delta_X, \ba^t)$ for some $f \in R\setminus \{0\}$.  Thus, again using Lemma \ref{Lemma:SkodaTypeTheorem}, we may assume that $\Delta_{Y/\Tt}$ and $\Delta_X$ are effective divisors.  Also note that, because we may localize at an element of $R \,\,  \cap \, \, \tau_b(Y; \Delta_{Y/\Tt}, (\ba \cdot \O_Y)^t)$,  we know $\Tt(\tau_b(Y; \Delta_{Y/\Tt}, (\ba \cdot \O_Y)^t)) \neq 0$.

Now, for every $e \geq 0$ and every map $\phi \in I_{\Delta_X, e} \subseteq \Hom_R(\Frob{R}{e}, R)$, $\phi$ has a $\Tt$-transpose ${\phi}_{\Tt}$ in $I_{\Delta_{Y/\Tt}, e} \subseteq \Hom_S(\Frob{S}{e}, S)$.  Also note that
\[
\begin{array}{rcl}
\Tt\Bigg({\phi}_{\Tt}\left(\FrobP{\ba^{\lceil t(p^e - 1) \rceil} \tau_b(Y; \Delta_{Y/\Tt}, (\ba \cdot \O_Y)^t)}{e}\right) \Bigg) &\subseteq& \Tt\left(\tau_b(Y; \Delta_{Y/\Tt}, (\ba \cdot \O_Y)^t)\right) \, \, .
\end{array}
\]
But since $\Tt \, \circ \, {\phi}_{\Tt} = \phi \,  \circ \, \Frob{\Tt}{e}$ by assumption, we obtain
\[
\begin{array}{l}
\phi \left(\Frobp{\ba^{\lceil t(p^e - 1) \rceil}  \Tt \left(\tau_b(Y;
      \Delta_{Y/\Tt}, (\ba \cdot \O_Y)^t)\right)}{e}\right) \\ \quad \qquad
\begin{array}{cl}
 =& \phi \left(\Tt^{1/p^e} \Frob{\left(\ba^{\lceil t(p^e - 1) \rceil} \tau_b(Y; \Delta_{Y/\Tt}, (\ba \cdot \O_Y)^t)\right)}{e}\right) \\
=& \Tt\left({\phi}_{\Tt}\left(\FrobP{\ba^{\lceil t(p^e - 1) \rceil} \tau_b(Y; \Delta_{Y/\Tt}, (\ba \cdot \O_Y)^t)}{e}\right) \right) \\
\subseteq& \Tt(\tau_b(Y; \Delta_{Y/\Tt}, (\ba \cdot \O_Y)^t)).
\end{array}
\end{array}
\]
This proves the Proposition since $\tau_b(X; \Delta_X, \ba^t)$ is the smallest non-zero ideal satisfying the same condition.
\end{proof}

We now give two proofs of the reverse inclusion. In the first, we provide an easy argument which is relatively intuitive but works only in the log $\bQ$-Gorenstein case where the index is not divisible by $p$.  The proof in the general case will follow immediately thereafter.

\begin{proposition}
\label{Prop:HardContainmentForQGorenstein}
Assume we are in the setting of \ref{set.TestIdealsSetting}.  Further suppose that $K_X + \Delta_X$ is $\bQ$-Cartier with index not divisible by $p$. Then
\[
\Tt\left(\tau_b(Y; \Delta_{Y/\Tt}, (\ba \cdot \O_Y)^t) \right) \subseteq \tau_b(X, \Delta_X, \ba^t).
\]
\end{proposition}
\begin{proof}
Again, the statement is local on $X$ so we assume $X = \Spec R$ for a local ring $R$ and that $Y = \Spec S$ is the spectrum of a semi-local ring $S$.
By hypothesis, there exists a \mbox{$p^{-e}$-linear} map $\phi \: \Frob{\O_X}{e} \rightarrow \O_X$ corresponding to $\Delta_X$ for some sufficiently large $e$.  As in the proof of the previous proposition, by using Lemma \ref{Lemma:SkodaTypeTheorem}, we may assume that $\Delta_{Y/\Tt}$, $\Delta_X$ and $\RamiT$ are effective divisors.  Set ${\phi}_{\Tt}$ to be the $\Tt$-transpose of $\phi$.

Now choose $c \in R\setminus \{0\}$ satisfying the condition of Lemma \ref{Lemma:ExistOfTestElts} for both affine triples $(X = \Spec R, \Delta_X, \ba^t)$ and $(Y = \Spec S, \Delta_{Y/\Tt}, (\ba \cdot \O_Y)^t )$.
Such an element exists by Remark~\ref{RemarkExistenceHelp}.

Again let $\phi^{n} \: \Frob{R}{ne} \to R$ and $\left( {\phi_{\Tt}}
\right)^n \: \Frob{S}{ne} \rightarrow S$ denote the $n$-th Frobenius iterates of $\phi$ and $\phi_{\Tt}$, respectively, as in (\ref{eqn:ComposeMapWithSelf}).  Then
\[
\begin{array}{rcl}
 \Tt \left(\tau_b(Y; \Delta_{Y/\Tt}, (\ba \cdot \O_Y)^t) \right)
&=& \Tt \left( \sum_{n \geq 0} \left( {\phi_{\Tt}} \right)^n \left( \Frobp{c \ba^{\lceil t(p^{ne} - 1) \rceil} S}{ne} \right) \right)\smallskip \\
&=& \sum_{n \geq 0} \phi^n \left( \Tt^{1/p^{ne}} \left( \Frobp{c \ba^{\lceil t(p^{ne} - 1) \rceil} S}{ne} \right) \right) \smallskip \\
&\subseteq& \sum_{n \geq 0} \phi^n \left(  \Frobp{c \ba^{\lceil t(p^{ne} - 1) \rceil} R}{ne} \right) \smallskip \\
&=&\tau_b(X, \Delta_X, \ba^t).
\end{array}
\]
Here the second line is due to  Corollary \ref{cor:trcommutes} and the third is deduced from the fact that $\Tt(S) \subseteq R$ and also that $\Tt$ is $R$-linear.
\end{proof}

We now prove the above inclusion in the general case (that is, without the log $\bQ$-Gorenstein assumption).  The proof strategy has similarities to \cite[Proposition 1.12]{HaraYoshidaGeneralizationOfTightClosure}.  First we need a lemma which says, roughly, that if we pick our ``test elements'' carefully, we can use round-down instead of round-up in the construction of the test ideal.

\begin{lemma}
\label{lem:rounddownTestIdeal}
Suppose we are in the setting of \ref{set.TestIdealsSetting} and further assume that $\Delta_X$ and $\Delta_{Y/\Tt}$ are effective and that $X = \Spec R$ and $Y = \Spec S$.  Fix an element $c \in R \setminus \{ 0 \}$ that satisfies Lemma \ref{Lemma:ExistOfTestElts} for both $(X, \Delta_X, \ba^t)$ and also for $(Y, \Delta_{Y/\Tt}, (\ba \cdot \O_Y)^t)$.  Choose $c' \in R\setminus \{0\}$ such that $\Div_X(c') \geq \Delta_X$ and $\Div_Y(c') \geq \Delta_{Y/\Tt}$.   Then we have
\begin{equation}
\label{eqn:AlternateTestIdealDesc1}
\tau_b(X; \Delta_X, \ba^t) = \sum_{e \geq 0} \left( \sum_{\phi \in J_{\Delta_X, e}} \phi\left(\Frobp{(cc') \ba^{\lceil t(p^e - 1) \rceil} }{e}\right) \right).
\end{equation}  Likewise, we also have
\begin{equation}
\label{eqn:AlternateTestIdealDesc2}
\tau_b(Y; \Delta_{Y/\Tt}, (\ba \cdot \O_Y)^t) = \sum_{e \geq 0} \left( \sum_{\phi \in J_{\Delta_{Y/\Tt}, e}} \phi\left(\Frobp{(cc') (\ba \cdot \O_{Y})^{\lceil t(p^e - 1) \rceil} }{e} \right) \right).
\end{equation}
\end{lemma}
\begin{remark} The difference between these statements and Proposition
  \ref{prop:ConstructionOfTestIdeal} is that here we sum over $\phi \in J_{\Delta, e}$ instead of $\phi \in I_{\Delta, e}$ (see Definition \ref{DefnIDeltaJDelta}).
\end{remark}

\begin{proof}
Simply observe that
\[
\begin{array}{rcl}
\tau_b(X; \Delta_X, \ba^t) &=& \sum_{e \geq 0} \left( \sum_{\phi\in I_{\Delta_X, e} } \phi\left(\Frobp{cc' \ba^{\lceil t(p^e - 1) \rceil}}{e} \right) \right) \smallskip \\
&\subseteq& \sum_{e \geq 0} \left( \sum_{\phi \in J_{\Delta_X, e}} \phi\left(\Frobp{cc' \ba^{\lceil t(p^e - 1) \rceil} }{e} \right) \right) \smallskip \\
&=& \sum_{e \geq 0} \left( \sum_{\phi \in (\Frob{c'}{e} J_{\Delta_X, e})} \phi\left(\Frobp{c \ba^{\lceil t(p^e - 1) \rceil} }{e}\right) \right) \smallskip \\
&\subseteq& \sum_{e \geq 0} \left( \sum_{\phi\in I_{\Delta_X, e} } \phi\left(\Frobp{c \ba^{\lceil t(p^e - 1) \rceil} }{e}\right) \right) \smallskip \\
&=& \tau_b(X; \Delta_X, \ba^t).
\end{array}
\]
This completes the proof of the first statement.  The second follows in similar fashion.
\end{proof}

\begin{proposition}
\label{Prop:HardTestIdealContainment}
Suppose we are in the setting of \ref{set.TestIdealsSetting}.   Then
\[
\Tt\left(\tau_b(Y; \Delta_{Y/\Tt}, (\ba \cdot \O_Y)^t) \right) \subseteq \tau_b(X, \Delta_X, \ba^t).
\]
\end{proposition}
\begin{proof}
The statement is local so we assume that $X$ is the spectrum of a local ring $R$ and that $Y$ is the spectrum of a semi-local ring $S$.  As in the proof of Proposition \ref{Prop:HardContainmentForQGorenstein}, using Lemma \ref{Lemma:SkodaTypeTheorem}, we assume that $\Delta_{Y/\Tt}$, $\Delta_X$ and $\RamiT$ are effective divisors.  Fix $c$ and $c'$ as in Lemma \ref{lem:rounddownTestIdeal}.

Consider the following diagram
\begin{equation}
\label{eqn:HardDiagramForTestIdealContainment}
\xymatrix{
\Hom_R(\Frob{R}{e}, R) \ar[r]^-{\varepsilon_X} & R \\
\Hom_S(\Frob{S}{e}, S) \ar[u]^{\eta} \ar[r]_-{\varepsilon_Y} & S \ar[u]_{\Tt}
}
\end{equation}
where $\varepsilon_X$ and $\varepsilon_Y$ are both evaluation at $cc'$, and $\eta$ is the composition
\[
\small
\xymatrix@C=12pt{
\eta \: \Hom_S(\Frob{S}{e}, S) \ar[r] & \Hom_R(\Frob{S}{e}, S) \ar[r]^-{\Tt} & \Hom_R(\Frob{S}{e}, R) \ar[r]^-{\kappa} & \Hom_R(\Frob{R}{e}, R)
}
\]
first of the forgetful map (every $S$-homomorphism is also an $R$-homomorphism), second with the application of $\Tt$ to the target, and lastly restriction of the source from $S^{1/p^{e}}$ to $R^{1/p^{e}}$.  One may easily verify the commutativity of the diagram in (\ref{eqn:HardDiagramForTestIdealContainment}).  The desired conclusion follows immediately from (\ref{eqn:AlternateTestIdealDesc1}), (\ref{eqn:AlternateTestIdealDesc2}), and (\ref{eqn:HardDiagramForTestIdealContainment}), together with the Claim \ref{clm.technicalProofInclusion} below.

\begin{claim}
\label{clm.technicalProofInclusion}
 $\eta(J_{\Delta_{Y/\Tt}, e}) \subseteq J_{\Delta_X, e} \subseteq
 \Hom_R(\Frob{R}{e}, R)$.
\end{claim}

To verify the claim, since $J_{\Delta_X, e}$ is a reflexive
$R$-module, it is sufficient to check over a height one prime of $R$.
Therefore, we may assume that $R$ is a DVR and that $S$ is a
semi-local Dedekind domain (in particular, both rings are principal ideal domains).  Choose $\Psi$ to be an $S$-module generator of $\Hom_R(S, R)$, $\Phi_S$ to be an $\Frob{S}{e}$-module generator of $\Hom_S(\Frob{S}{e}, S)$, and $\Phi_R$ to be an $\Frob{R}{e}$-module generator of $\Hom_R(\Frob{R}{e}, R)$.  We also have that $\Psi \, \circ \,  \Phi_S (\blank)= \Phi_R \, \circ \,  \Psi^{1/p^e}(u^{1/p^e} \cdot \blank)$ for some unit $u \in S$, since both $\Psi \, \circ \, \Phi_S$ and $\Phi_R \, \circ \, \Psi^{1/p^e}$ generate $\Hom_R(S^{1/p^e}, R)$ as an $S^{1/p^e}$-module; see for example \cite[Lemma 3.9]{SchwedeFAdjunction} or \cite[Appendix F]{KunzKahlerDifferentials}.  We may write $\Tt = \Psi(d \cdot \blank)$ where $d$ is a defining equation for $\RamiT$.  Write $\Delta_X = t \Div_X(a)$ where $a$ is a uniformizer for $R$.   Then $\Delta_{Y/\Tt} = t \Div_Y(a) - \Div_Y(d)$.  Also write $t \Div_Y(a) =\sum_i t_i \Div_Y(a_i)$ where the $\Div_Y(a_i)$ are prime divisors.

Notice that for any $e > 0$, $b_e = \left( \prod a_i^{\lfloor t_i (p^e- 1) \rfloor}\right) / a^{\lfloor t(p^e - 1) \rfloor}$ is an element of $S$.  This follows from the relation $\lfloor (p^{e}-1) t \rfloor \divisor_{Y}(a) \leq \lfloor (p^{e}-1) t \divisor_{Y}(a) \rfloor$.
Now, fix $\phi \in J_{\Delta_{Y/\Tt}, e}$ and write
\[
\phi(\blank) = \Phi_S \left(  \Frobp{ s {1 \over d^{p^e - 1} } (\prod
    a_i^{\lfloor t_i(p^e - 1) \rfloor}) }{e} \cdot \blank \right)
\]
where $s$ is an element of $S$.
Then for any $x^{1/p^e} \in R^{1/p^e}$ we have
\[
\begin{array}{rcl}
\eta(\phi)(x^{1/p^e}) &=& \Tt\left( \Phi_S\left(  \Frobp{xs {1 \over d^{p^e - 1} } \prod_i a_i^{\lfloor t_i(p^e - 1) \rfloor} }{e}\right) \right) \smallskip \\
&=&  \Psi\left(d \Phi_S\left(  \Frobp{ xs {1 \over d^{p^e - 1} } \prod_i a_i^{\lfloor t_i(p^e - 1) \rfloor} }{e} \right) \right) \smallskip \\
&=& \Psi\left( \Phi_S( \Frobp{ x s d \prod_i a_i^{\lfloor t_i(p^e - 1) \rfloor} }{e} ) \right) \smallskip \\
&=& \Phi_R \left( \Psi^{1/p^{e}}\left( \Frobp{ u x s d \prod_i a_i^{\lfloor t_i(p^e - 1) \rfloor} }{e} \right) \right) \smallskip \\
&=& \Phi_R \left( \Frobp{ x a^{\lfloor t(p^e - 1) \rfloor} }{e} \cdot \Psi^{1/p^{e}}\left( \Frobp{usdb_e}{e} \right) \right) \, \, .\\
\end{array}
\]
This proves the claim since we have $\Psi^{1/p^{e}}\left( \Frobp{usdb_e}{e} \right) \in R^{1/p^{e}}$.
\end{proof}

The previous two propositions combine into our main theorem for test ideals.

\begin{theorem}
\label{thm:TraceTestIdealFormula}
Suppose we are in the setting of \ref{set.TestIdealsSetting}.   Then
\[
\Tt\left(\tau_b(Y; \Delta_{Y/\Tt}, (\ba \cdot \O_Y)^t) \right) = \tau_b(X; \Delta_X, \ba^t).
\]
\end{theorem}

In the case that $\pi$ is separable and $\Tt$ is the trace map, we obtain the following.

\begin{corollary}
\label{cor:TraceTestIdealFormula}
Suppose we are in the setting of \ref{set.TestIdealsSetting}, and furthermore that $\pi$ is separable and $\Tt = \Tr_{Y/X}$ is the trace map.  Then
\[
\Tr\left(\tau_b(Y; \pi^* \Delta_{X} - \Ram_{\pi}, (\ba \cdot \O_Y)^t) \right) = \tau_b(X; \Delta_X, \ba^t).
\]
\end{corollary}

\subsection{Test ideals, $F$-regularity, and $F$-purity under separable finite maps}
\label{sec:testidealproofs}

While the previous section explained how the test ideal transforms
under arbitrary finite surjective maps, more can often be said if $\pi$ is separable.  In this
case, we leverage the fact that maps in $\Hom_R(F^e_* R, R)$ extend to
maps in $\Hom_S(F^e_* S, S)$.  Therefore, throughout this section, we work under the following assumptions.

\begin{setting}
\label{set.SeparableTestIdealsSetting}
Let $\pi \: Y \rightarrow X$ be a finite separable morphism of
$F$-finite normal schemes
and $(X, \Delta_X, \ba^t)$ a triple.
Let $\Ram_{\pi}$ denote the ramification divisor of $\pi$ and $\Tr =
\Tr_{Y/X} \: \pi_{*}\O_{Y} \to \O_{X}$ the trace map.  We set
\[
\Delta_Y = \pi^*\Delta_X - \Ram_{\pi}
\]
and will consider the triple $(Y, \Delta_{Y},  (\ba \cdot \O_{Y})^{t})$.
\end{setting}

We begin with a description of how sharp $F$-purity behaves under finite maps.

\begin{theorem}
\label{thm:TraceSurjectiveImpliesFPureBehaves}
Suppose we are in the setting of \ref{set.SeparableTestIdealsSetting} and also assume that $\Delta_X$ and $\Delta_Y$ are effective.
Further suppose that $K_X + \Delta_X$ is $\bQ$-Cartier with index not divisible by $p > 0$ and that $\Tr = \Tr_{Y/X} \: \pi_{*}\O_{Y} \to \O_{X}$ is surjective.  Then $(X, \Delta_X, \ba^t)$ is sharply $F$-pure if and only if $(Y, \Delta_Y, (\ba \cdot \O_Y)^t)$ is sharply $F$-pure.
\end{theorem}
\begin{proof}
The statement is local on $X$ and so we assume that $X = \Spec R$ where $R$ is a local ring and $Y = \Spec S$ for a semi-local ring $S$.  There exists an $e_0 > 0$ such that for all $e = ne_0$ we have that $\Hom_R(\Frobp{R( (p^e - 1)\Delta_{X})}{e}, R)$ is free as an $\Frob{R}{e}$-module.  If $\phi_{X} \: R^{1/p^e} \to R$ is a generator, then $\Delta_{\phi_{X}} = \Delta_{X}$ and $I_{\Delta_{X}, \, e}$ consists entirely of the $R^{1/p^{e}}$-multiples of $\phi_{X}$.
Furthermore, it follows that $\Hom_S(\Frobp{S( (p^e - 1)\Delta_Y)}{e}, S) \cong \pi^* \O_X( (1-p^e)(K_X + \Delta_X)) \cong \pi^* \O_X \cong \O_Y$ is free as an
$\Frob{S}{e}$-module.
Thus, if we let $\phi_{Y} \: S^{1/p^e} \to S$ be the map extending $\phi_X$ (and thus a generator of $\Hom_S(\Frobp{S( (p^e - 1)\Delta_Y)}{e}, S) $ corresponding to $\Delta_Y$), similar considerations apply to $I_{\Delta_{X}, \, e}$.

First let us suppose $(X, \Delta_{X}, \ba^{t})$ is sharply $F$-pure.  From the description of $I_{\Delta_{X}, \, e}$ above and the definition of sharp $F$-purity, after possibly enlarging $e$ we must have
\[
1 \in \phi_{X}\left(\Frobp{ \ba^{\lceil (p^e - 1) \rceil} }{e} \right) \subseteq \phi_{Y}\left(\Frobp{ (\ba S)^{\lceil (p^{e}-1) \rceil}}{e}\right)
\]
and hence $(Y, \Delta_{Y}, (\ba \cdot \O_{Y})^{t})$ is sharply $F$-pure.

Conversely, suppose that $(Y, \Delta_{Y}, ( \ba \cdot \O_{Y})^{t})$ is sharply $F$-pure.  Making $e$ larger if necessary, we may assume that $\phi_Y((\ba \cdot \O_Y)^{\lceil t(p^e - 1) \rceil}) = \O_Y$ since it is true after localizing at each maximal ideal of $S$ for sufficiently large and divisible $e$.
 We have the commutative diagram
\[
\xymatrix{
\Frob{R}{e} \ar[r]^-{{\phi}_X} & R \\
\Frob{S}{e}  \ar[u]^-{\Frob{\Tr}{e}} \ar[r]_-{\phi_Y} & S \ar[u]_-{\Tr} \\
}
\]
and the statement immediately follows from the surjectivity of trace, as \[
(\Frob{\Tr}{e})\left(\Frobp{(\ba \cdot \O_Y)^{\lceil t(p^e - 1) \rceil}}{e}\right) = \Frobp{\ba^{\lceil t(p^e - 1) \rceil}}{e} \, \, .
\]
\end{proof}

\begin{remark}
 Example \ref{ex:nottame1} demonstrates that the surjectivity of the
 trace map is essential for the reverse implication in Theorem
 \ref{thm:TraceSurjectiveImpliesFPureBehaves}.  The forward
 direction always holds regardless of the sujectivity of the trace map.
\end{remark}

In the case that the trace map is surjective, we also obtain the following
formula.  As noted in the introduction, it mirrors the transformation
rule \eqref{eqn:multIdealFact} for the multiplier ideal in
characteristic zero.

\begin{corollary}
\label{cor:IfTraceSurjectiveThenIntersection}
Suppose we are in the setting of \ref{set.SeparableTestIdealsSetting}
and further suppose that the trace map $\Tr = \Tr_{Y/X} \: \pi_{*}\O_{Y} \to \O_{X}$ is surjective.   Then
\[
K(X) \cap \pi_* \tau_b(Y; \Delta_Y, (\ba \cdot \O_Y)^t) = \tau_b(X; \Delta_X, \ba^t)
\]
where $K(X)$ is the function field of $X$.  If in addition $\Delta_{X} \geq 0$, then
\[
\O_{X} \cap \pi_* \tau_b(Y; \Delta_Y, (\ba \cdot \O_Y)^t) = \tau_b(X; \Delta_X, \ba^t) \, \, .
\]
\end{corollary}
\begin{proof}
The statement is local so we may assume that $X$ is the spectrum of a local ring $R$, and that $Y$ is the spectrum of a semi-local ring $S$.  As in the proof of Proposition \ref{Prop:HardContainmentForQGorenstein}, by using Lemma \ref{Lemma:SkodaTypeTheorem}, we assume that both $\Delta_Y$ and $\Delta_X$ are effective divisors.
Since $\Tr \: S \to R$ is surjective and $R$-linear, we see that
\[
R \cap \tau_b(Y; \Delta_Y, (\ba \cdot \O_Y)^t) \subseteq \Tr(\tau_b(Y; \Delta_Y, (\ba \cdot \O_Y)^t)) = \tau_{b}(X;\Delta_{X}, \ba^{t}) \, \, .
\]
On the other hand, using an argument similar to Proposition \ref{Prop:EasyTestIdealContainment}, one easily verifies (or sees Theorem~\ref{thm:reallynonoptimalextensionresult} below) that $R \,  \cap \, \tau_b(Y; \Delta_Y, (\ba \cdot \O_Y)^t) \supseteq \tau_b(X; \Delta_X, \ba^t)$.  Indeed, since every $\phi \in I_{\Delta_{X}, \, e}$ has an extension $\bar{\phi} \in I_{\Delta_{Y}, \, e}$, it follows that
$\phi\left( \Frobp{\ba^{\lceil t(p^e - 1) \rceil} (R \cap \tau_{b}(Y;\Delta_{Y}, (\ba \cdot \O_{Y})^{t}))}{e}\right) \subseteq R \cap \tau_{b}(Y;\Delta_{Y}, (\ba \cdot \O_{Y})^{t})$ as we have $\bar{\phi}\left(\Frobp{\ba^{\lceil t(p^e - 1) \rceil} (\tau_{b}(Y; \Delta_{Y}, (\ba \cdot \O_{Y})^{t}))}{e}\right) \subseteq \tau_{b}(Y; \Delta_{Y}, (\ba \cdot \O_{Y})^{t})$.
\end{proof}

\begin{corollary}
\label{cor:IfTraceSurjectsThenStrongFRegularity}
Suppose we are in the setting of \ref{set.SeparableTestIdealsSetting}
and further suppose that the trace map $\Tr = \Tr_{Y/X} \: \pi_{*}\O_{Y} \to \O_{X}$ is surjective.  Then
$(X, \Delta_X \geq 0, \ba^t)$ is strongly $F$-regular if and only if $(Y, \Delta_Y \geq 0, (\ba \cdot \O_Y)^t)$ is strongly $F$-regular.
\end{corollary}
\begin{proof}
This follows immediately since $(X, \Delta_X, \ba^t)$ (respectively $(Y, \Delta_Y, (\ba \cdot \O_Y)^t)$) is strongly $F$-regular if and only if $\tau_b(X, \Delta_X, \ba^t) = \O_X$ (respectively $\tau_b(Y; \Delta_Y, (\ba \cdot \O_Y)^t) = \O_Y$).
\end{proof}

\begin{remark}
Analogous to Corollary \ref{cor:IfTraceSurjectsThenStrongFRegularity},
the property of having Kawamata log terminal (klt) singularities is
equivalent for appropriately
related triples under finite surjective morphisms in characteristic zero
\cite[Proposition 5.20(4)]{KollarMori}.
\end{remark}

\begin{example}
\label{ex:WithoutTraceSurjectiveThingsAreFalse}
Without the hypothesis that the trace map is surjective, Corollaries \ref{cor:IfTraceSurjectiveThenIntersection} and \ref{cor:IfTraceSurjectsThenStrongFRegularity} are false.  To see this, work in the setting of Example \ref{ex:nottame1}.  In this case, one has an extension of rings $k[[x^2(x^3+1)]] = R \subseteq S = k[[x]]$ and a map $\phi \: R^{1/p^e} \to R$ that extends to a map $\bar{\phi} \: S^{1/p^e} \to S$ which generates $\Hom_S(S^{1/p^e}, S)$ as an $S^{1/p^e}$-module.  Therefore $\tau_b(S, \Delta_{\bar{\phi}}) = S$.  On the other hand one can check that $(x^2(x^3+1))^2 = \tau_b(R; \Delta_{\phi}) \neq R$.  Because this example considers one-dimensional regular rings, the test ideal coincides with the characteristic $p > 0$ multiplier ideal.  Therefore equation \eqref{eqn:multIdealFact} also does not hold for multiplier ideals even for separable maps in characteristic $p > 0$.
\end{example}

We conclude with a different type of containment result about test ideals.  However, in the presence of essentially any ramification at all, this result nearly always fails to be optimal, see Example \ref{ex:NonOptimalExtension} below.

\begin{theorem}
  \label{thm:reallynonoptimalextensionresult}
Suppose we are in the setting of \ref{set.SeparableTestIdealsSetting} where we additionally assume $\Delta_{X}$ is effective.   Then
\[
\tau_b(X; \Delta_X, \ba^t) \cdot \O_{Y} \subseteq \tau_b(Y; \Delta_Y, (\ba \cdot \O_Y)^t)  \, \, .
\]
\end{theorem}

\begin{proof}
  The statement is local so we may assume that $X$ is the spectrum of a local ring $R$ and that $Y$ is the spectrum of a semi-local ring $S$.  As before, it is harmless to assume that $\Delta_Y$ is effective.  Choose $c \in R\setminus \{0\}$ satisfying the condition of Lemma \ref{Lemma:ExistOfTestElts} for both the affine triple $(X = \Spec R, \Delta_X, \ba^t)$ and the affine triple $(Y = \Spec S, \Delta_Y, (\ba \cdot \O_Y)^t )$.  Then, again recalling that every $\phi \in I_{\Delta_{X}, \, e}$ has an extension $\bar{\phi} \in I_{\Delta_{Y}, \, e}$, we see
\[
\begin{array}{rcl}
  \tau_{b}(X;\Delta_{X}, \ba^{t}) \cdot \O_{Y} & = & \left( \sum_{e\geq 0} \left( \sum_{\phi \in I_{\Delta_{X}, \, e}} \phi \left( \Frob{\left(c \ba^{\lceil (p^{e}-1)t \rceil}\right)}{e} \right) \right) \right) \cdot \O_{Y} \smallskip \\
& \subseteq & \sum_{e\geq 0} \left( \sum_{\phi \in I_{\Delta_{X}, \, e}} \bar{\phi} \left( \Frob{\left( c (\ba \cdot \O_{Y})^{\lceil (p^{e}-1)t \rceil}\right)}{e} \right) \right) \smallskip \smallskip \\
& \subseteq & \sum_{e\geq 0} \left( \sum_{\psi \in I_{\Delta_{Y}, \, e}} \psi \left( \Frob{\left( c (\ba \cdot \O_{Y})^{\lceil (p^{e}-1)t \rceil}\right)}{e} \right) \right) \smallskip \\
& = &\tau_{b}(Y;\Delta_{Y}, (\ba \cdot \O_{Y})^{t})
\end{array}
\]
as desired.
\end{proof}

As mentioned, this result is far from optimal.

\begin{example}
\label{ex:NonOptimalExtension}
Consider the extension of rings $R = k[x^2] \subseteq k[x] = S$ where $\text{char } k \neq 2$.  Set $X = \Spec R$, $Y = \Spec S$, and $\Delta_X = \Div_X(x^2)$ so that $\Delta_Y = \Div_Y(x^2) - \Div_Y(x) = \Div_Y(x)$.  In this case, $\tau_b(X; \Delta_X) \cdot \O_Y = x^2 \cdot \O_Y$ but $\tau_b(Y; \Delta_Y) = x \cdot \O_Y$.
\end{example}

\begin{corollary}
  \label{cor:nonoptimalfregextension}
Suppose we are in the setting of \ref{set.SeparableTestIdealsSetting}  where we additionally assume $\Delta_{X}$ and $\Delta_{Y}$ are effective.      If $(X, \mbox{$\Delta_X \geq 0$}, \ba^t)$ is strongly $F$-regular, then  $(Y, \mbox{$\Delta_Y \geq 0$}, \mbox{$(\ba \cdot \O_Y)^t$})$ is as well.
\end{corollary}

\section{On the surjectivity of the trace map}
\label{sec:SurjectivityOfTrace}

In the previous section, we saw that the surjectivity of the trace map
was useful for describing the behavior of the test ideal under
separable finite morphisms.  We therefore have the following natural question.

\begin{question}
\label{quest:WhenDoesTraceSurject}
 Suppose that $R \subseteq S$ is a generically separable module-finite
 inclusion of \mbox{$F$-finite} normal domains with $K = \Frac R$ and $L = \Frac S$.  When is it true that $\Tr_{L/K}(S) = R$?
\end{question}

\begin{remark}
The condition that $\Tr_{L/K}(S) = R$ has been recently coined \emph{cohomologically tamely ramified} in the Galois case by Kerz and Schmidt; see \cite[Claim 1, Theorem 6.2]{KerzSchmidtOnDifferentNotionsOfTameness}.
\end{remark}
The rest of this section will be devoted to giving some partial
answers to Question \ref{quest:WhenDoesTraceSurject}.  The methods of
this section can also be adapted to other (fixed) maps $\Tt : S \to
R$, and we leave it to the reader to flush out precise statements and details.  We begin with a negative answer: the contrapositive of Theorem \ref{thm:TraceSurjectiveImpliesFPureBehaves}.  This will be followed by several positive results.

\begin{proposition}
 Suppose that $R \subseteq S$ is a generically separable module-finite
 inclusion of $F$-finite normal domains of characteristic $p > 0$ with $K = \Frac R$ and $L = \Frac S$.
 Suppose that $\phi \: R^{1/p^e} \rightarrow R$ is a map that extends
 to a map $\bar{\phi} \: S^{1/p^e} \to S$.  Further suppose that $(R,
 \Delta_{\phi})$ is not sharply $F$-pure (\ie $\phi$ is \emph{not}
 surjective) but $(S, \Delta_{\bar{\phi}})$ is sharply $F$-pure (\ie
 $\bar{\phi}$ is surjective).  Then the trace map is not surjective,
 \ie $\Tr_{L/K}(S) \neq R$.
\end{proposition}

\begin{proposition}
\label{prop:purePlusEtaleInCodim1}
 Suppose that $R \subseteq S$ is a module-finite pure (a
   module-finite extension of rings $R \to S$ is pure if and only if
   it splits as a map of $R$-modules) inclusion of normal domains
 that is \etale in codimension one, and set $K = \Frac R$ and $L =
 \Frac S$.  Then the trace map is surjective, \ie $\Tr_{L/K}(S) = R$.
\end{proposition}
\begin{proof}
 Since $R \subseteq S$ is \etale in codimension one, by Proposition
 \ref{prop:traceram}, the $S$-submodule of $\Hom_R(S, R)$ generated by
 $\Tr = \Tr_{L/K}|_{S}$ agrees with $\Hom_R(S,R)$ in codimension one.  But both modules are reflexive (since $R$ is normal) so $\Tr$ generates $\Hom_R(S, R)$ as an $S$-module.

By the purity hypothesis, there exists a surjective $\phi \in
\Hom_R(S, R)$.  We know that there exists some $s \in S$ such
that $\phi( \blank) = \Tr(s \cdot \blank)$ since $\Tr$ generates $\Hom_R(S, R)$.  Therefore, $\Tr$ must also be  surjective as desired.
\end{proof}

\begin{remark}
The formula in Corollary \ref{cor:IfTraceSurjectiveThenIntersection} for $\Delta_{X}\geq 0$, namely
\[
 \O_X \cap \tau_b(Y; \Delta_Y, (\ba \cdot \O_Y)^t) = \tau_b(X; \Delta_X, \ba^t) \, \, ,
\]
was known to hold previously in the case of an extension of $F$-finite normal domains that was both pure and \etale in codimension one; see \cite[Theorem 3.3]{HaraTakagiOnAGeneralizationOfTestIdeals}.  We see from above that the hypotheses of Corollary \ref{cor:IfTraceSurjectiveThenIntersection} are already satisfied in this setting.  Therefore, Theorem~\ref{thm:TraceTestIdealFormula} should be viewed as a strict generalization of the previously known formula.
\end{remark}


\begin{theorem}
\label{thm:surjectivityOfTraceViaTestIdeals}
 Suppose that $R \subseteq S$ is a module-finite inclusion of $F$-finite normal domains with corresponding map of schemes $\pi \: Y \rightarrow X$.  Further suppose that $\Delta_X$ is an effective divisor on $X$ satisfying the following conditions
\begin{itemize}
 \item[(i)]  $\Delta_Y := \pi^* \Delta_X - \Ram_{\pi}$ is effective,
 \item[(ii)]  $(X, \Delta_X)$ is strongly $F$-regular
\end{itemize}
Then trace is surjective, \ie $\Tr_{Y/X}(S) = R$.
\end{theorem}
\begin{proof}
 Since $(X, \Delta_X)$ is strongly $F$-regular we know that $\tau_b(R; \Delta_X) = R$.  Then $$\Tr_{Y/X}(\tau_b(S; \Delta_Y)) = \tau_b(R; \Delta_X) = R$$ by Corollary \ref{cor:TraceTestIdealFormula}.  Since $\tau_b(S; \Delta_Y) \subseteq S$ by (i), the proof is complete.
\end{proof}

The following corollaries to Theorem \ref{thm:surjectivityOfTraceViaTestIdeals} are easier to digest and extend known results.

\begin{corollary}
\label{cor:surjectivityOfTraceOverStronglyFRegular}
Suppose that $R \subseteq S$ is a module-finite inclusion of
$F$-finite normal domains which is \etale in codimension one with $K = \Frac R$ and $L = \Frac S$.  If $R$
is strongly \mbox{$F$-regular}, then $\Tr_{L/K}(S) = R$.
\end{corollary}
\begin{proof}
This result immediately follows from Theorem \ref{thm:surjectivityOfTraceViaTestIdeals}.  Alternately, one can use the fact that every strongly $F$-regular ring is a splinter ring (see \cite{HochsterHunekeTightClosureOfParameterIdealsAndSplitting}) so that $R \subseteq S$ splits, and then we can apply Proposition \ref{prop:purePlusEtaleInCodim1}.
\end{proof}

Example \ref{ex:artinsD4examplecover} below demonstrates that, in Corollary \ref{cor:surjectivityOfTraceOverStronglyFRegular}, one cannot weaken the condition that $R$ is strongly $F$-regular to the condition that $R$ is $F$-pure.  In particular, condition (ii) from Theorem \ref{thm:surjectivityOfTraceViaTestIdeals} \emph{cannot} be replaced by the condition that ``$(X, \Delta_X)$ is sharply $F$-pure.''

We now explain how the $F$-pure threshold can be used to show that the
trace map is surjective.  Given a strongly $F$-regular ring $R$ and a
Weil divisor $B$ on $X = \Spec R$, the \emph{$F$-pure threshold} of
$(X,B)$ is defined to be the real number
$$\fpt(X, B) := \sup \{ \lambda \, \, | \, \, (X, \lambda B) \text{ is
  strongly $F$-regular} \} \, \, ;$$
 see \cite{TakagiWatanabeFPureThresh} and \cite{MustataTakagiWatanabeFThresholdsAndBernsteinSato}.  It is a characteristic $p > 0$ analog of the log canonical threshold.

\begin{corollary}
Suppose that $R \subseteq S$ is a module-finite inclusion of $F$-finite
normal domains with corresponding map of schemes $\pi \: Y \rightarrow
X$ tamely ramified in codimension one.  Further suppose that $R$ is
strongly $F$-regular and that the reduced branch locus $B$ of $\pi$ has $F$-pure threshold $c := c(X, B)$ inside $X$.  If for every Weil divisor $F$ on $Y$ with ramification index $n$ over $X$, we have that $c > {n - 1 \over n}$, then $\Tr_{Y/X}(S) = R$.
\end{corollary}
\begin{proof}
Set $\Delta_X$ to be the $\bQ$-divisor $(c -\varepsilon) B$ for some $1 \gg \varepsilon > 0$.  The relation between the ramification index and the log canonical threshold is exactly what is needed to satisfy conditions (i) and (ii) from Theorem \ref{thm:surjectivityOfTraceViaTestIdeals}.
\end{proof}

Sometimes, in the case that $B$ has sufficiently nice singularities (regardless of how it is embedded into $X$), we can also show that $\Tr_{Y/X}(S) = R$.

\begin{corollary}
Suppose that $R \subseteq S$ is a module-finite inclusion of $F$-finite normal domains with corresponding map of schemes $\pi \: Y \rightarrow X$ and suppose that $R$ is strongly $F$-regular and \mbox{$\bQ$-Gorenstein} with index not divisible by $p > 0$.  Suppose that $B$ is the reduced branch locus of $\pi$.  Further suppose that $B \subseteq D$ where $D$ is an integral divisor satisfying all of the following:
\begin{itemize}
 \item $D$ is Cartier,
 \item $D$ is a reduced scheme with $F$-pure singularities, and
 \item $D$ is S2 and Gorenstein in codimension one. (This last condition follows, for example, if $D$ is normal or $R$ is Gorenstein.)
\end{itemize}
If $R \subseteq S$ has tame ramification in codimension one, then $\Tr_{Y/X}(S) = R$.
\end{corollary}
\begin{proof}
 Since $D$ is Cartier and $R$ is $\bQ$-Gorenstein with index not divisible by $p$, it follows that $(X, D)$ is $F$-pure by $F$-inversion of adjunction. To verify this, see \cite[Theorem 4.9]{HaraWatanabeFRegFPure} for $F$-inversion of adjunction in the case where $D$ is a normal scheme.  Alternately, in the case that $D$ is non-normal, the $F$-purity of the pair $(X, D)$ follows from \cite[Proposition 7.2]{SchwedeFAdjunction}.  But then $(X, (1-\varepsilon)D)$ is strongly $F$-regular for all $0 \leq \varepsilon < 1$ by \cite[Proposition 2.2(5)]{TakagiWatanabeFPureThresh}.

Since $R \subseteq S$ has tame ramification in codimension one, it is straightforward to verify that $\pi^* (1 - \varepsilon) D \geq \Ram_{\pi}$ for all $0 < \varepsilon \ll 1$.  Thus both conditions (i) and (ii) from Theorem \ref{thm:surjectivityOfTraceViaTestIdeals} are satisfied for $\Delta_X = (1 - \varepsilon)D$.
\end{proof}

As a special case, we immediately obtain:

\begin{corollary}
\label{cor:EasyTraceSurjectivityViaFPurity}
 Suppose that $R \subseteq S$ is a module-finite inclusion of $F$-finite normal domains which is tamely ramified in codimension one.  Consider the corresponding map of schemes $\pi \: Y \rightarrow X$.  Further suppose that $R$ is regular and that the reduced branch locus of $\pi$ is a divisor with $F$-pure singularities (for example, if the branch locus is a simple normal crossings divisor).  Then $\Tr_{Y/X}(S) = R$.
\end{corollary}

\begin{remark}
 Brian Conrad has pointed out to us that a special case of the
 previous corollary, where the branch locus is a simple normal
 crossings divisor, can be obtained as follows.  For simplicity,
 assume that $R$ is a complete regular local ring and that $S$ is a
 complete normal local ring, and also that both have algebraically closed residue fields.  Fix the branch locus to be $\Div(f_1 \cdot \ldots \cdot f_r)$.  One may replace $S$ by a further extension and using \cite[Corollary 5.3, Exp. XIII]{SGA1} (a version of Abhyankar's Lemma)  may assume that $S$ is of the form
 \[
 S = R[T_1, \ldots, T_r]/(T_1^{n_1} - f_1, \ldots, T_r^{n_r} - f_r)
 \]
 where the $n_i$ are integers not divisible by $p$.  But then, in this
 situation $S$ is free as an $R$-module and the degree of $S$ over $R$ is not divisible by $p$.  Thus $\Tr_{Y/X}$ is clearly surjective since $1 \in S$ is sent to a unit in $R$.
\end{remark}

We conclude the body of the paper with the following cautionary
example, originally due to Artin.   It is a finite integral extension
of normal domains which has non-surjective trace map but is \etale in
codimension one.  In particular, the example shows that Theorem \ref{thm:surjectivityOfTraceViaTestIdeals}(ii) cannot be
substantially weakened.  Also note that this example is a log terminal
singularity which is $F$-pure but not strongly $F$-regular.

Notice that Examples \ref{ex:nottame1} and \ref{ex:WithoutTraceSurjectiveThingsAreFalse} demonstrate that the relation \eqref{eqn:multIdealFact} cannot hold for multiplier ideals in positive characteristic.  On the other hand, the example below shows that the multiplier ideal in positive characteristic cannot
satisfy the kinds of transformation rules for test ideals from Corollary \ref{cor:TraceTestIdealFormula}, even in the separable case.

\begin{example}[Double cover of a certain $\mathrm{D}_{4}$ singularity with wild ramification in characteristic 2; see \cite{ArtinWildlyRamifiedZ2Actions}]
\label{ex:artinsD4examplecover}
Consider $X = \Spec R$ where
\[
R = \mathbb{F}_{2} [[ x,y,z]] / \langle z^{2} + xyz + xy^{2} + x^{2}y \rangle \, .
\]
One can verify that this is in fact a rational double point of type
$\mathrm{D}_{4}$; see \cite{LipmanRationalSingularitiesWithApps}.  As
a sufficient theory of resolutions of singularities holds for surfaces
in any characteristic, we know that $X$ is log terminal or that the
multiplier ideal $\mJ(X) = \O_{X} = R$ (\cf
\cite{TuckerIntegrallyClosedIdealsOnLTSurfaces}) is trivial.  In \cite{ArtinWildlyRamifiedZ2Actions}, it is shown that the extension ring
\[
S = R[u,v]/ \langle u^{2} + xu + x, v^{2} + yv + y, z + xv + yu \rangle
\]
gives rise to a finite morphism $\pi \: Y = \Spec S \to X$ of degree two which is {\'e}tale in codimension one.  Furthermore, one may verify directly that
\[
S = \mathbb{F}_{2}[[u,v]]
\]
is in fact a power series ring in the variables $u,v$. Thus, in particular, $\mJ(Y) = \O_{Y} = S$.  Artin also shows that $R$ can be viewed as the invariant subring of $S$ for an action of $\mathbb{Z}/2\mathbb{Z}$.
Using his notation and denoting the action by the nontrivial element with a bar, we have
\[
\bar{u} = u + \frac{1}{1+u} u^{2} \qquad \bar{v} = v + \frac{1}{1+v} v^{2} \, \, .
\]
He shows $R$ is the complete subring of $S$ determined by the elements
\[
x = u + \bar{u} \qquad y = v + \bar{v} \qquad z = u \bar{v} + \bar{u} v \, \, .
\]
Direct calculation then gives
\[
\Tr(u) = x \qquad \Tr(v) = y \qquad \Tr(uv) = xy + z
\]
so we see that the $\Tr \: S \to R$ has image $\langle x,y,z \rangle$ and is not surjective.  In particular, we have
\[
\Tr(\mJ(Y)) \neq \mJ(X) \, \, .
\]

On the other hand, $R$ is $F$-pure by Fedder's criterion; see
\cite{FedderFPureRat}.  It also fails to be strongly $F$-regular, and it follows that $\tau_b(R) = \langle x, y, z \rangle$ since $R$ has an isolated singularity at the maximal ideal.  We see that
\[
\Tr(\mJ(Y)) = \Tr(\tau_b(Y)) = \tau_b(R) = \langle x, y, z \rangle
\,\, ,
\]
which one may also verify directly.
This example gives credence to the maxim that test ideals in positive characteristic have more desirable behavior and properties than multiplier ideals.
\end{example}

\section{Further speculation about test ideals and trace maps}
\label{sec:furtherSpeculation}

In the previous section we explored when the trace map
is surjective, but many open problems remain.  Based on the previous section, the following converse to Corollary \ref{cor:IfTraceSurjectiveThenIntersection} is a natural question.

\begin{question}
Suppose we are in the setting of \ref{set.SeparableTestIdealsSetting} and that
\[
K(X) \cap \pi_* \tau_b(Y; \Delta_Y, (\ba \cdot \O_Y)^t) = \tau_b(X; \Delta_X, \ba^t)
\]
holds (perhaps for every choice of $\Delta_X$).  Does it then follow that $\Tr_{Y/X} \: \pi_{*}\O_{Y} \to \O_{X}$ is surjective?
\end{question}

Let us assume now in addition to \ref{set.SeparableTestIdealsSetting}
that $X= \Spec R$ and $Y = \Spec S$ are affine.  As we saw in the
previous section, in many cases $\Tr \: S \to R$ is necessarily
surjective when $R$ is strongly $F$-regular.  It would be interesting
to know if similar properties might in fact characterize strong $F$-regularity.

\begin{question}
If one knows that $\Tr \: S \to R$ is surjective for some carefully chosen (but at this point unspecified) collection of finite separable extensions $R \subseteq S$, does this imply that $R$ is strongly $F$-regular?
\end{question}

\noindent
This question is likely related to various conjectures about the
equivalence of splinter rings and weak $F$-regularity; see
\cite{HochsterHunekeTightClosureOfParameterIdealsAndSplitting} and
\cite{SinghQGorensteinSplinter} for precise statements and further details.


We conclude with a question involving the test ideal and extension under flat morphisms.
Given a finite surjective morphism $\pi \: Y \to X$ such that $\pi$ is (locally) pure, \etale in codimension one, and flat, it is known that
\[
\tau_b(X) \cdot \O_Y = \tau_b(Y) \, \, ;
\]
see \cite[Theorem 3.3]{HaraTakagiOnAGeneralizationOfTestIdeals}, \cite{BravoSmithBehaviorOfTestIdealsUnderSmooth} and \cite{HochsterHunekeFRegularityTestElementsBaseChange}.  Thus, one is left with the following natural question.

\begin{question}
Given a finite flat surjective morphism $\pi \: Y \to X$ and a
$\Q$-divisor $\Delta_{X}$ on $X$,  is there a way in which to extend
$\tau_b(X, \Delta_X)$ to $Y$ and thereby obtain $\tau_b(Y,
\Delta_{Y})$ for a suitably chosen $\Q$-divisor $\Delta_{Y}$ on $Y$
without assuming that $\pi$ is \etale in codimension 1?
If possible, what is the appropriate way to define $\Delta_{Y}$?
\end{question}

In a sense, one can view the above question as asking for an optimal
version of Theorem~\ref{thm:reallynonoptimalextensionresult}.
However, we are at present unsure whether it is reasonable to expect such a statement to hold, let alone what an appropriate formulation might be.  Coincidentally, we are unaware of any such statements for multiplier ideals even in characteristic zero (for which similar questions could be posed as well).

\bibliographystyle{skalpha}
\bibliography{CommonBib}  

\end{document}